\documentclass[12pt,a4paper]{article}
\usepackage{amsfonts,amssymb,amsmath,amsthm}
\usepackage[T1]{fontenc}
\usepackage[english]{babel} 
\usepackage[utf8]{inputenc}
\usepackage{a4wide}

 \newcommand{\R}{\mathbb R}
 \newcommand{\N}{\mathbb N}

\newcommand{\Svarcprim}{{S'(\R)}}

\begin{document}

\newtheorem{theorem}{Theorem}[section]
\newtheorem{proposition}[theorem]{Proposition}
\newtheorem{corollary}[theorem]{Corollary}
\newtheorem{lemma}[theorem]{Lemma}
\newtheorem{definition}[theorem]{Definition}
\newtheorem{example}[theorem]{Example}
\newtheorem{remark}[theorem]{Remark}

\title{\bf Stochastic evolution equations with multiplicative noise}

\author{
Tijana~Levajkovi\'c\thanks{University of Belgrade, Serbia, t.levajkovic@sf.bg.ac.rs} ,
Stevan~Pilipovi\'c\thanks{Faculty of Sciences, University of Novi Sad, Serbia, stevan.pilipovic@dmi.uns.ac.rs} ,
Dora~Sele\v si\thanks{Faculty of Sciences, University of Novi Sad, Serbia, dora.selesi@dmi.uns.ac.rs} ,
Milica \v Zigi\'c\thanks{Faculty of Sciences, University of Novi Sad, Serbia, milica.zigic@dmi.uns.ac.rs}
	} 
\date{}
\maketitle


\bigskip

\begin{abstract}
We study parabolic stochastic partial differential equations
(SPDEs), driven by two types of operators: one linear closed
operator generating a $C_0-$semigroup and one linear bounded
operator with Wick-type multiplication, all of them set in the
infinite dimensional space framework of white noise analysis. We
prove existence and uniqueness of solutions for this class of SPDEs.
In particular, we also treat the stationary case when the
time-derivative is equal to zero.
\end{abstract}


 \section{Introduction and definitions} \label{sec:intro}

We  consider a stochastic Cauchy problem of the form
\begin{equation}\label{nasaJNAuvod}\begin{split}\frac{d}{dt}U(t,x,\omega) &= {\mathbf A}U(t,x,\omega) + {\mathbf B}\lozenge U(t,x,\omega) + F(t,x,\omega)\\
U(0,x,\omega)&=U^0(x,\omega),\end{split}\end{equation} where
$t\in(0,T]$, $\omega\in\Omega$, and $U(t,\cdot,\omega)$ belongs to some Banach space
$X$. The operator $\mathbf A$ is densely defined, generating a
$C_0-$semigroup and $\mathbf B$ is a linear bounded
operator which combined with the Wick product $\lozenge$ introduces
convolution-type perturbations into the equation. All stochastic processes
are considered in the setting of Wiener-It\^o chaos expansions. A
comprehensive explanation of the action of the operators $\mathbf A$ and
$\mathbf B$ in this framework will be provided in Section \ref{0.3}.

Our investigations in this paper are inspired by \cite{Boris} where
the authors provide a comprehensive analysis of equations of the
form $$\frac{d}{dt}u(t,x,\omega) = {\mathbf A} u(t,x,\omega) +
\delta({\mathbf M}u(t,x,\omega)) = {\mathbf A} u(t,x,\omega) + \int
{\mathbf M} u(t,x,\omega)\lozenge W(x,\omega)\;dx,$$ where $\delta$
denotes the Skorokhod integral and $W$ denotes the spatial white
noise process. In Proposition \ref{Dora2} we prove that for every
operator ${\mathbf M}$ there exists a corresponding operator
$\mathbf B$ such that ${\mathbf B}\lozenge u=\delta({\mathbf M}u)$.
On the other hand, the class of operators $\mathbf B$ is much
larger. This holds also for the class of operators $\mathbf A$ we
consider (a comprehensive analysis of all operators is given in
Section \ref{0.4}). Thus, we extend the results of \cite{Boris} and
\cite{Boris2} to a more general class of stochastic differential
equations which are driven by two linear multiplicative operators: $\mathbf A$ acting
with ordinary multiplication, while $\mathbf B\lozenge$ is
acting with the convolution-type Wick product.

We have studied elliptic SPDEs, in particular the  stochastic
Dirichlet problem of the form ${\mathbf L}\lozenge u+f=0$ in our
previous papers \cite{JSAA}, \cite{ps}, \cite{DP2}. As a conclusion
to this series of papers we study parabolic SPDEs of the form
\eqref{nasaJNA}.  Such equations also include as a special case equations of
the form $\frac{d}{dt}u={\mathbf L}u+f$ and $\frac{d}{dt}u={\mathbf
L}\lozenge u+f$, where $L$ is a strictly elliptic second order
partial differential operator. These equations describe the heat conduction in
random media (inhomogeneous and anisotropic materials), where the
properties of the material are modeled by a positively definite
stochastic matrix.

Other special cases of \eqref{nasaJNA} include the heat equation
with random potential $\frac{d}{dt}u=\Delta u + {\mathbf B}\lozenge
u$, the Schr\"odinger  equation $(i\hbar)\frac{d}{dt}u=\Delta u + {\mathbf B}\lozenge
u+f$,
 the transport equation $\frac{d}{dt}u=\frac{d^{2}}{dx^2}u +
W\lozenge\frac{d}{dx}u$ driven by white noise as in \cite{Proske},
the generalized Langevin equation $ \frac{d}{dt}u= {\mathbf
J}u+{\mathbf C}(Y'),$ where $Y$ is a L\'evy process, ${\mathbf J}$
the infinitesimal generator of a $C_0-$semigroup and ${\mathbf C}$ a
bounded operator, which was studied in \cite{Dejv}, as well as the
equation $\frac{d}{dt} u= {\mathbf L} u + W \lozenge u,$ where
${\mathbf L}$ is a strictly elliptic partial differential operator
as studied in \cite{Cat2} and \cite{yuhu}.

Equations of the form $\frac{d}{dt}u= {\mathbf A} u+ {\mathbf B}W$
were also studied in
 \cite{Irina} and
\cite{Irina2}, where ${\mathbf A}$ is not necessarily generating a
$C_0-$semigroup, but an $r$-integrated or a convolution semigroup.

In order to solve \eqref{nasaJNA} we apply the method of
Wiener-It\^o chaos expansions, also known as the propagator method.
With this method we reduce the SPDE to an infinite triangular system
of PDEs, which can be solved by induction. Summing up all
coefficients of the expansion and proving convergence in an
appropriate weight space, one obtains the solution of the initial
SPDE.

We also consider the case of stationary equations $\mathbf A U+\mathbf B\lozenge U+F=0$.  In particular, elliptic SPDEs have been studied in \cite{JSAA}, \cite{Boris2}, \cite{ps} and \cite{DP2}. With the method of chaos expansions one can also treat hyperbolic SPDEs \cite{grci} and SPDEs with singularities \cite{FundSol}. One of its advantages is that it provides explicit solutions in terms of a series expansion, which can be easily implemented also to numerical approximations and computational simulations.

\subsection{$C_0-$semigroups}\label{0.1}

We recall some well-known facts which will be used in the sequel
(see \cite{paz}). Let $X$ be a Banach space.
If $B$ is a bounded linear operator on $X$ and $A$ is the
infinitesimal generator of a  $C_0-$semigroup $\{T_t\}_{ t\geq 0}$
satisfying
            $||T_t||_{L(X)}\leq M e^{w t},\;t\geq 0, \mbox{ for some }\;M, w>0,$
then $A+B$ is the infinitesimal generator of a $C_0-$semigroup
$\{S_t\}_{t\geq 0},$ on $X$ satisfying $$\|S_t\|_{L(X)}\leq M
e^{(w+M\|B\|_{L(X)})t},\;t\geq 0.$$

Let $u(0)=u^0\in D=Dom(A)$  and $f\in C([0,\infty),X)$. Recall that
$u:[0,T]\to X$ is a \emph{(classical) solution} on $[0,T]$ to
\begin{equation}\label{pert cauchy pr}\frac{d}{dt}u(t)=Au(t)+f(t),\; t\in(0,T],\quad u(0)= u^0,\end{equation} if $u$ is continuous
on $[0,T]$, continuously differentiable on $(0,T]$, $u(t)\in
D,\;t\in (0,T]$ and the equation is satisfied on $(0,T]$. If $f\in
L^1((0,T), X)$, then
$u(t)=T_t u^0+\int_0^t T_{t-s}f(s)ds, t\in [0,T]$ belongs to
$C([0,T],X)$, and it is called a \emph{mild solution.} Clearly, a
mild solution that is continuously differentiable on $(0,T]$ is a
classical solution.

Let $f\in L^1((0,T),X)\cap C((0,T],X)$ and $v(t)=\int_{0}^t
T_{t-s}f(s)ds, t\in[0,T].$ The initial value problem  has a solution
$u$ for every $u^0\in D$ if one of the following conditions is
satisfied (see \cite{paz}):
            \begin{itemize}

            \item[(i)] $v$ is continuously differentiable on $(0,T)$.

            \item[(ii)] $v(t)\in D$ for $0<t\leq T$ and $Av(t)$ is continuous on $(0,T]$.

            \end{itemize}
If the initial value problem has a solution on $[0,T]$ for some
$u^0\in D$, then $v(t)$ satisfies both (i) and (ii). Note that if $f\in C^1([0,T],X)$ then conditions (i) and (ii) are fulfilled. Moreover, if $f\in C^1([0,T],X)$ and $u^0\in D(A)$, then for the solution $u$ of (\ref{pert cauchy pr}) we have that $u\in C^1([0,T],X)$ and $\frac{d}{dt}u(0)=Au^0+f(0)$.


\subsection{Generalized stochastic processes}\label{0.2}

Denote by $(\Omega, \mathcal{F}, P) $ the Gaussian white noise
probability space $(S'(\mathbb{R}), \mathcal{B}, \mu), $ where
$S'(\mathbb{R})$ denotes the space of tempered distributions,
$\mathcal{B}$ the Borel sigma-algebra generated by the weak topology
on $S'(\mathbb{R})$ and $\mu$ the Gaussian white noise measure corresponding to  the
  characteristic function
\begin{equation*}\label{BM theorem}
\int_{S'(\mathbb{R})} \,  e^{{i\langle\omega, \phi\rangle}}
d\mu(\omega) = \exp \left [-\frac{1}{2} \|\phi\|^2_{L^2(\mathbb{R})}\right], \quad \quad\phi\in  S(\mathbb{R}),
\end{equation*}
given by the Bochner-Minlos theorem.

We recall the notions related to $L^2(\Omega,\mu)$ (see \cite{HOUZ})
where $\Omega=S'(\R)$ and $\mu$ is Gaussian white noise measure.
Define the set of multi-indices $\mathcal I$ to be $(\mathbb N_0^\mathbb N)_c$,
i.e. the set of sequences of non-negative integers which have only
finitely many nonzero components. Especially, we denote by $\mathbf 0=(0,0,0,\ldots)$ the multi-index with all entries equal to zero. The length of a multi-index is $|\alpha|=\sum_{i=1}^\infty\alpha_i$ for $\alpha=(\alpha_1,\alpha_2,\ldots)\in\mathcal I$, and it is always finite. Similarly, $\alpha!=\prod_{i=1}^\infty\alpha_i!$, and all other operations are also carried out componentwise. We will use the convention that $\alpha-\beta$ is defined if $\alpha_n-\beta_n\geq0$ for all $n\in\mathbb N$, i.e., if $\alpha-\beta\geq\mathbf 0$, and leave $\alpha-\beta$ undefined if $\alpha_n<\beta_n$ for some $n\in\mathbb N$.

The
Wiener-It\^o theorem (sometimes also referred to as the
Cameron-Martin theorem) states that one can define an orthogonal
basis $\{H_\alpha\}_{\alpha\in\mathcal I}$ of $L^2(\Omega,\mu)$,
where $H_\alpha$ are constructed by  means of Hermite orthogonal
polynomials $h_n$ and Hermite functions $\xi_n$,
$$H_\alpha(\omega)=\prod_{n=1} ^\infty h_{\alpha_n}(\langle\omega,\xi_n\rangle),\quad \alpha=(\alpha_1,\alpha_2,\ldots, \alpha_n\ldots)\in\mathcal I,\quad \omega\in\Omega=S'(\mathbb{R}).$$
Then, every $F\in L^2(\Omega,\mu)$ can be represented via the so
called \emph{chaos expansion}
$$F(\omega)=\sum_{\alpha\in\mathcal I} f_\alpha H_\alpha(\omega), \quad \omega\in S'(\mathbb{R}),\quad\sum_{\alpha\in\mathcal I} |f_\alpha|^2\alpha!<\infty,\quad f_\alpha\in\R,\quad\alpha\in\mathcal I.$$

Denote by $\varepsilon_k=(0,0,\ldots, 1, 0,0,\ldots),\;k\in \N$ the
multi-index with the entry 1 at the $k$th place. Denote by $\mathcal
H_1$ the subspace of $L^2(\Omega,\mu)$, spanned by the polynomials
$H_{\varepsilon_k}(\cdot)$, $k\in\mathbb N$. The subspace $\mathcal
H_1$ contains Gaussian stochastic processes, e.g. Brownian motion is
given by the chaos expansion $B(t,\omega) = \sum_{k=1}^\infty
\int_0^t \xi_k(s)ds\;H_{\varepsilon_k}(\omega).$

Denote by $\mathcal H_m$ the $m$th order chaos space, i.e. the
closure of the linear subspace spanned by the orthogonal polynomials
$H_\alpha(\cdot)$ with $|\alpha|=m$, $m\in\mathbb N_0$. Then the
Wiener-It\^o chaos expansion states that
$L^2(\Omega,\mu)=\bigoplus_{m=0}^\infty \mathcal H_m$, where
$\mathcal H_0$ is the set of constants in $L^2(\Omega,\mu)$.


It is well-known that the time-derivative of Brownian motion (white
noise process) does not exist in the classical sense. However,
changing the topology on $L^2(\Omega,\mu)$ to a weaker one, T. Hida
\cite{Hida} defined spaces of generalized random variables
containing the white noise as a weak derivative of the Brownian
motion. We refer to \cite{Hida}, \cite{HOUZ} for white noise
analysis (as an infinite dimensional analogue of the Schwartz theory
of deterministic generalized functions).

Let $(2\mathbb N)^{\alpha}=\prod_{n=1}^\infty (2n)^{\alpha_n},\quad
\alpha=(\alpha_1,\alpha_2,\ldots, \alpha_n,\ldots)\in\mathcal I.$ We
will often use the fact that the series $\sum_{\alpha\in\mathcal
I}(2\mathbb N)^{-p\alpha}$ converges  for $p>1$. Define the Banach
spaces
$$(S)_{1,p} =\{F=\sum_{\alpha\in\mathcal I}f_\alpha {H_\alpha}\in L^2(\Omega,\mu):\;  \|F\|^2_{(S)_{1,p}}= \sum_{\alpha\in\mathcal I}(\alpha!)^2 |f_\alpha|^2(2\mathbb N)^{p\alpha}<\infty\},\quad p\in\mathbb N_0.$$
Their topological dual spaces are given by
$$(S)_{-1,-p} =\{F=\sum_{\alpha\in\mathcal I}f_\alpha {H_\alpha}:\;  \|F\|^2_{(S)_{-1,-p}}= \sum_{\alpha\in\mathcal I}|f_\alpha|^2(2\mathbb N)^{-p\alpha}<\infty\},\quad p\in\mathbb N_0.$$
The Kondratiev space of generalized random variables is $(S)_{-1}
=\bigcup_{p\in\mathbb N_0}(S)_{-1,-p}$ endowed with the inductive
topology. It is the strong dual of $(S)_{1} =\bigcap_{p\in\mathbb
N_0}(S)_{1,p}$, called the Kondratiev space of test random variables
which is endowed with the projective topology.
Thus,
$$(S)_{1} \subseteq L^2(\Omega,\mu) \subseteq
(S)_{-1}$$ forms a Gelfand triplet.

The time-derivative of the Brownian motion exists in the generalized
sense and belongs to the Kondratiev space $(S)_{-1,-p}$ for
$p\geq\frac5{12}$. We refer to it as to \emph{white noise}
and its formal expansion is given by
 $W(t,\omega) =
\sum_{k=1}^\infty \xi_k(t)H_{\varepsilon_k}(\omega).$

We extended in \cite{GRPW} the definition of stochastic processes
also to processes of the chaos expansion form
$U(t,\omega)=\sum_{\alpha\in\mathcal I}u_\alpha(t)
{H_\alpha}(\omega)$, where the coefficients $u_\alpha$ are elements
of some Banach space $X$. We say that $U$ is an \emph{$X$-valued generalized stochastic process}, i.e. $U(t,\omega)\in X\otimes (S)_{-1}$ if there exists $p>0$ such that $\|U\|_{X\otimes(S)_{-1,-p}}^2=\sum_{\alpha\in\mathcal I}\|u_\alpha\|_X^2(2\mathbb N)^{-p\alpha}<\infty$.

 The \emph{Wick product} of stochastic
processes $F=\sum_{\alpha\in\mathcal I}f_\alpha H_\alpha,
G=\sum_{\beta\in\mathcal I}g_\beta H_\beta\in X\otimes(S)_{-1}$
 is  $$F\lozenge G = \sum_{\gamma\in\mathcal
I}\sum_{\alpha+\beta=\gamma}f_\alpha g_\beta H_\gamma =
\sum_{\alpha\in\mathcal I}\sum_{\beta\leq \alpha} f_\beta
g_{\alpha-\beta} H_\alpha,$$ and the $n$th Wick power is defined by
$F^{\lozenge n}=F^{\lozenge (n-1)}\lozenge F$, $F^{\lozenge 0}=1$.
 Note that
$H_{n\varepsilon_k}=H_{\varepsilon_k}^{\lozenge n}$ for $n\in\mathbb
N_0$, $k\in\mathbb N$.

For example, let $X=C^k[0,T]$, $k\in\mathbb N$. In \cite{ps} we
proved that differentiation of a stochastic process can be carried
out componentwise in the chaos expansion, i.e. due to the fact that
$(S)_{-1}$ is a nuclear space it holds that
$C^k([0,T],(S)_{-1})=C^k[0,T]\otimes(S)_{-1}$. This means that a
stochastic process $U(t,\omega)$ is $k$ times continuously
differentiable if and only if all of its coefficients $u_\alpha(t)$,
$\alpha\in\mathcal I$ are in $C^k[0,T]$.

The same holds for Banach space valued stochastic processes i.e.
elements of  $C^k([0,T],X)\otimes(S)_{-1}$, where $X$ is an
arbitrary Banach space. By the nuclearity of $(S)_{-1}$, these
processes can be regarded as elements of the tensor product space
$$C^k([0,T],X\otimes
(S)_{-1})=C^k([0,T],X)\otimes(S)_{-1}=\bigcup_{p=0}^{\infty}C^k([0,T],X)\otimes
(S)_{-1,-p}.$$

\section{Stochastic operators}\label{0.3}

\begin{definition} Let $X$ be a Banach space and $\mathbf O: X\otimes(S)_{-1}\rightarrow
X\otimes(S)_{-1}$ an operator acting on the space of stochastic
processes. We will call $\mathbf O$ to be a \emph{coordinatewise operator} if there
exists a family of operators $o_\alpha:X\rightarrow X$,
$\alpha\in\mathcal I$, such that $\mathbf O(\sum_{\alpha\in\mathcal
I}f_\alpha H_\alpha)=\sum_{\alpha\in\mathcal
I}o_\alpha(f_\alpha)H_\alpha$ for all $F=\sum_{\alpha\in\mathcal
I}f_\alpha H_\alpha\in X\otimes(S)_{-1}$.
\end{definition}
Clearly, not all operators are coordinatewise, for example $\mathbf
O(F)=F^{\lozenge2}$ can not be written in this form.
\begin{definition}
The subclass of  \emph{simple coordinatewise operators} consists of
 operators for which $o_\alpha=o_\beta=o$,
$\alpha,\beta\in\mathcal I$, that is, they can be written in form of
$\mathbf O(\sum_{\alpha\in\mathcal I}f_\alpha
H_\alpha)=\sum_{\alpha\in\mathcal I}o(f_\alpha)H_\alpha$ for some
operator $o:X\rightarrow X$.
\end{definition}
 For example, the operator of differentiation \cite{ps} and
the Fourier transform \cite{FundSol}  are simple coordinatewise operators.
The Ornstein-Uhlenbeck operator is a coordinatewise operator  but it is not a simple coordinatewise operator.

Note that even if all $o_\alpha$, $\alpha\in\mathcal I$,
are bounded linear operators, the coordinatewise operator $\mathbf O$ itself does
not need to be bounded. If $o_\alpha$, $\alpha\in\mathcal I$, are
uniformly bounded by some $C>0$, then $\mathbf O$ is also a bounded operator. This
follows from
\begin{equation*}\begin{split}\|\mathbf O(F)\|^2_{X\otimes(S)_{-1,-p}}&\leq\sum_{\alpha\in\mathcal
I}\|o_\alpha\|^2_{L(X)}\|f_\alpha\|_X^2(2\mathbb N)^{-p\alpha}\\
&\leq
C^2\sum_{\alpha\in\mathcal I}\|f_\alpha\|_X^2(2\mathbb
N)^{-p\alpha}=C^2\|F\|^2_{X\otimes (S)_{-1,-p}}<\infty,\end{split}\end{equation*} for $F\in X\otimes (S)_{-1,-p}$.

This condition is sufficient, but not necessary, and can be loosened
by the embedding $(S)_{-1,-p}\subseteq (S)_{-1,-q}$, $q\geq p$.

\begin{lemma}
Let $\mathbf O$ be a coordinatewise operator for which all
$o_\alpha$, $\alpha\in\mathcal I$, are
polynomially bounded i.e. $\|o_\alpha\|_{L(X)}\leq R(2\mathbb
N)^{r\alpha}$ for some $r,R>0$. Then, there exists $q\geq p$ such that $\mathbf O: X\otimes (S)_{-1,-p}\rightarrow X\otimes (S)_{-1,-q}$ is bounded.
\end{lemma}

\begin{proof}
Let $q\geq p+2r$. Then,

\begin{align*}\|\mathbf O(F)\|^2_{X\otimes(S)_{-1,-q}}&\leq R^2\sum_{\alpha\in\mathcal
I}(2\mathbb N)^{2r\alpha}\|f_\alpha\|_X^2(2\mathbb N)^{-q\alpha} =R^2
\sum_{\alpha\in\mathcal I}\|f_\alpha\|_X^2(2\mathbb
N)^{-(q-2r)\alpha}\\ &\leq R^2
\sum_{\alpha\in\mathcal I}\|f_\alpha\|_X^2(2\mathbb
N)^{-p\alpha} =  R^2 \|F\|^2_{X\otimes(S)_{-1,-p}}<\infty.\end{align*} Thus,
$\|\mathbf O\|_{L(X)\otimes(S)_{-1}}\leq R$.
\end{proof}

Note that the condition $\|o_\alpha\|_{L(X)}\leq R(2\mathbb
N)^{r\alpha}$ for some $r,R>0$ is actually
equivalent to stating that there exists $r>0$ such that
$\sum_{\alpha\in\mathcal I}\|o_\alpha\|^2_{L(X)}(2\mathbb
N)^{-r\alpha}<\infty$.

\vspace{1cm}

Throughout the paper we will consider the equation
\begin{equation}\begin{split}\frac{d}{dt}U(t,\omega) &= {\mathbf A}U(t,\omega) + {\mathbf B}\lozenge U(t,\omega) +
F(t,\omega),\quad t\in(0,T],\omega\in\Omega,\\U(0,\omega)&=U^0(\omega),\end{split}\label{nasaJNA}\end{equation} where both operators $\mathbf A$ and $\mathbf B$ are
assumed to be coordinatewise operators, i.e. composed out of a family of
operators $\{A_\alpha\}_{\alpha\in\mathcal I}$,
$\{B_\alpha\}_{\alpha\in\mathcal I}$, respectively. The operators
$A_\alpha$, $\alpha\in\mathcal I$, are assumed to be infinitesimal
generators of $C_0-$semigroups with a common domain $D$ dense in $X$
and the action of  $\mathbf A$ is given by $\mathbf
A(U)=\sum_{\alpha\in\mathcal I}A_\alpha(u_\alpha)H_\alpha$, for
$U=\sum_{\alpha\in\mathcal I}u_\alpha H_\alpha\in Dom(\mathbf
A)\subseteq D\otimes (S)_{-1}$, where
$$Dom(\mathbf A)=\{U=\sum_{\alpha\in\mathcal I}u_\alpha H_\alpha\in
D\otimes (S)_{-1}:\;\exists p_U>0,\; \sum_{\alpha\in\mathcal
I}\|A_\alpha(u_\alpha)\|^2_X(2\mathbb N)^{-p_U\alpha}<\infty\}.$$ The
operators $B_\alpha$, $\alpha\in\mathcal I$, are assumed to be
bounded and linear  on $X$, and the action of the operator
$\mathbf B\lozenge:X\otimes(S)_{-1}\to X\otimes(S)_{-1}$ is defined
by
$$\mathbf B\lozenge (U)=\sum_{\alpha\in \mathcal{I}}\sum_{\beta\leq
\alpha}B_\beta (u_{\alpha-\beta}) H_\alpha = \sum_{\gamma\in
\mathcal{I}}\sum_{\alpha+\beta=\gamma} B_\alpha(u_\beta) H_\gamma.$$

In the next two lemmas we provide two sufficient conditions that ensure the operator $\mathbf B\lozenge$ to be well-defined. Both conditions are actually equivalent to the fact that $B_\alpha$, $\alpha\in\mathcal I$, are polynomially bounded, but they provide finer estimates on the stochastic order (Kondratiev weight) of the domain and codomain of $\mathbf B\lozenge$.

\begin{lemma}
If the operators $B_\alpha$, $\alpha\in\mathcal I$,  satisfy $\sum_{\alpha\in\mathcal
I}\|B_\alpha\|^2_{L(X)}(2\mathbb N)^{-r\alpha}<\infty$, then
$\mathbf B\lozenge$ is well-defined as a mapping $\mathbf B\lozenge: X\otimes(S)_{-1,-p}\rightarrow X\otimes(S)_{-1,-(p+r+m)}$, $m>1$.
\end{lemma}

\begin{proof}
For $U\in
X\otimes(S)_{-1,-p}$ and $q=p+r+m$ we have
\begin{equation*}\begin{split}\sum_{\gamma\in \mathcal{I}}\|
\sum_{\alpha+\beta=\gamma} B_\alpha(u_\beta) \|_X^2& (2\mathbb
N)^{-q\gamma}\leq\sum_{\gamma\in
\mathcal{I}}\Big[\sum_{\alpha+\beta=\gamma}
\|B_\alpha\|_{L(X)}\|u_\beta\|_X\Big]^2 (2\mathbb
N)^{-(p+r+m)\gamma}\\
&=\sum_{\gamma\in \mathcal{I}}(2\mathbb
N)^{-m\gamma}\left(\sum_{\alpha+\beta=\gamma}
\|B_\alpha\|_{L(X)}^2(2\mathbb N)^{-r\gamma}\right)\left(
\sum_{\alpha+\beta=\gamma}
\|u_\beta\|_X^2 (2\mathbb N)^{-p\gamma}\right)\\
&\leq  M\left(\sum_{\alpha\in \mathcal{I}}
\|B_\alpha\|_{L(X)}^2(2\mathbb N)^{-r\alpha}\right)\left(
\sum_{\beta\in \mathcal{I}} \|u_\beta\|_X^2 (2\mathbb
N)^{-p\beta}\right)<\infty,
\end{split}\end{equation*} where $M=\sum_{\gamma\in \mathcal{I}}(2\mathbb
N)^{-m\gamma}<\infty$, for $m>1$.
\end{proof}

\begin{lemma}
If the operators $B_\alpha$, $\alpha\in\mathcal I$,  satisfy  $\sum_{\alpha\in\mathcal
I}\|B_\alpha\|_{L(X)}(2\mathbb N)^{-\frac{r}{2}\alpha}<\infty$, for some $r>0$, then $\mathbf B\lozenge$ is well-defined as a mapping $\mathbf B\lozenge: X\otimes(S)_{-1,-r}\rightarrow X\otimes(S)_{-1,-r}$.
\end{lemma}

\begin{proof}
For $U\in
X\otimes(S)_{-1,-r}$, we have by the generalized Minkowski inequality that
\begin{eqnarray*}\sum_{\gamma\in \mathcal{I}}\|
\sum_{\alpha+\beta=\gamma} B_\alpha(u_\beta) \|_X^2 (2\mathbb
N)^{-r\gamma}&\leq&\sum_{\gamma\in
\mathcal{I}}\Big[\sum_{\alpha+\beta=\gamma}
\|B_\alpha\|_{L(X)}\|u_\beta\|_X\Big]^2 (2\mathbb
N)^{-r\gamma}\\
&\leq&\sum_{\gamma\in
\mathcal{I}}\Big[\sum_{\alpha+\beta=\gamma}
\|B_\alpha\|_{L(X)}(2\mathbb N)^{-\frac{r}{2}\alpha}\|u_\beta\|_X (2\mathbb N)^{-\frac{r}{2}\beta}\Big]^2\\
&\leq&\left(\sum_{\alpha\in
\mathcal{I}} \|B_\alpha\|_{L(X)}(2\mathbb N)^{-\frac{r}{2}\alpha}\right)^2\sum_{\beta\in
\mathcal{I}} \|u_\beta\|_X^2 (2\mathbb N)^{-r\beta}<\infty.
\end{eqnarray*}
\end{proof}

\subsection{Special cases and relationship to other works}\label{0.4}

Some of the most important operators of stochastic calculus are the
operators of the Malliavin calculus. We recall their
definitions in the generalized $\Svarcprim$ setting \cite{AADM}.

\begin{itemize}
\item[$\bullet $] The \emph{Malliavin derivative}, $\mathbb D$, as a stochastic gradient in the direction of white noise, is a linear and continuous mapping $\mathbb{D}:
X\otimes (S)_{-1} \rightarrow X\otimes \Svarcprim \otimes(S)_{-1}$
given by
$$
\mathbb{D} u=\sum_{\alpha \in \mathcal{I}}\sum_{k\in \mathbb{N}}\,
\alpha_k \, u_\alpha\, \otimes \, \xi_k\, \otimes
  {H}_{\alpha-\varepsilon_k},\quad\mbox{ for } u=\sum_{\alpha\in\mathcal
I}u_\alpha\otimes H_\alpha.
$$ In terms of quantum theory it corresponds to the annihilation operator
reducing the order of the chaos space ( $\mathbb D: \mathcal
H_m\rightarrow \mathcal H_{m-1}$).
\item[$\bullet $]  The \emph{Skorokhod integral}, $\delta$, as an extension of the It\^o integral to non-anticipating processes, is a linear and continuous mapping
$\delta: X\otimes \Svarcprim \otimes(S)_{-1}\rightarrow
X\otimes(S)_{-1}$  given by
$$ \delta(F )=\sum_{\alpha\in
\mathcal{I}}\sum_{k\in \mathbb{N}}  f_{\alpha}\otimes v_{\alpha,k}
\otimes H_{\alpha+\varepsilon_k},\quad\mbox{ for }
F=\sum_{\alpha\in\mathcal I} f_\alpha\otimes\left(\sum_{k\in
\mathbb{N}}v_{\alpha, k}\, \xi_k\right)\otimes H_\alpha.$$ It is the
adjoint operator of the Malliavin derivative and in terms of quantum
theory it corresponds to the creation operator increasing the order
of the chaos space ($\delta: \mathcal H_m\rightarrow \mathcal
H_{m+1}$).
\item[$\bullet $]  The \emph{Ornstein-Uhlenbeck operator}, $\mathcal R$, as the composition of the previous ones $\delta\circ\mathbb D$, is the stochastic analogue of the Laplacian.
It is a linear and continuous mapping $\mathcal
R:X\otimes(S)_{-1}\rightarrow X\otimes(S)_{-1}$ given by $$\mathcal
R(u) = \sum_{\alpha\in\mathcal I}|\alpha|u_\alpha\otimes
H_\alpha,\quad\mbox{ for } u=\sum_{\alpha\in\mathcal
I}u_\alpha\otimes H_\alpha.
$$ In terms of quantum theory it corresponds to the number operator.
It is a selfadjoint operator $\mathcal R:\mathcal H_m\rightarrow
\mathcal H_m$ with eigenvectors equal to the basis elements
$H_\alpha$, $\alpha\in\mathcal I$, i.e. $\mathcal
R(H_\alpha)=|\alpha|H_\alpha$, $\alpha\in\mathcal I$. Thus, Gaussian
processes with zero expectation are the only fixed points for the
Ornstein-Uhlenbeck operator.
\end{itemize}
Clearly, the Ornstein-Uhlenbeck operator is a coordinatewise operator,
while the Malliavin derivative and the Skorokhod integral are not coordinatewise operators.

The Ornstein-Uhlenbeck operator is the infinitesimal generator of
the semigroup $T_t=e^{t\mathcal R}$, $t\geq0$, given by
$T_t(u)=\sum_{\alpha\in\mathcal I}e^{-|\alpha|t}u_\alpha\otimes
H_\alpha$, for $u=\sum_{\alpha\in\mathcal I}u_\alpha\otimes
H_\alpha\in X\otimes(S)_{-1}$.

It is also closely connected to the Ornstein-Uhlenbeck process. The
Ornstein-Uhlenbeck process is the solution of the SDE
$du(t,\omega)=-u(t,\omega)dt+dB(t,\omega)$,
$u(0,\omega)=u_0(x,\omega)$, and it is given by
$u(t,\omega)=e^{-t}u_0(\omega)+\int_0^t e^{t-s}dB(s,\omega)$. It is
a Markov process with transition semigroup $\{T_t\}_{t\geq0}$
\cite{Bog}. The solution of the generalized heat equation
$\frac{d}{dt}u+\mathcal R(u)=0$, $u(0)=u_0$, is given by
$u=T_t(u_0)$, i.e. $u(t,x)=(T_tu_0)(x)$ and
$(T_t\varphi)(x)=E(\varphi(u(t,x))$ for any $\varphi\in C_b(\mathbb
R)$ and $u$ is the Ornstein-Uhlenbeck process.

\vspace{5mm}

Now we turn to our equation
\begin{equation}\frac{d}{dt}U(t,\omega) = {\mathbf A}U(t,\omega) + {\mathbf B}\lozenge U(t,\omega) +
F(t,\omega),\label{SPDE-1}
\end{equation}
where $\mathbf A$ and $\mathbf B$ are coordinatewise operators  as
described in Section \ref{0.3}, composed out of a family of
operators $\{A_\alpha\}_{\alpha\in\mathcal I}$,
$\{B_\alpha\}_{\alpha\in\mathcal I}$, respectively, where $A_\alpha$
are infinitesimal generators on $X$ and $B_\alpha$ are bounded
linear operators on $X$, both families being polynomially bounded,
and their actions given by
\begin{equation}\label{delovanje A} {\mathbf A}U=\sum_{\alpha\in\mathcal
I}A_\alpha(u_\alpha)H_\alpha,\qquad \mbox{ for} \; \; U=\sum_{\alpha\in\mathcal I}u_\alpha H_\alpha,\end{equation}
\begin{equation}\label{delovanje B} {\mathbf B}\lozenge U =\sum_{\alpha\in \mathcal{I}}\sum_{\beta\leq
\alpha}B_\beta (u_{\alpha-\beta}) H_\alpha,\qquad \mbox{ for}\;\; U=\sum_{\alpha\in\mathcal I}u_\alpha H_\alpha. \end{equation}  Some important
special cases include the following:

 \begin{enumerate}
\item[I)] Special cases for $\mathbf A$:
            \begin{enumerate}
            \item[1)] $\mathbf A$ is a simple coordinatewise operator, i.e. $A_\alpha=A, \alpha\in\mathcal I,$ where $A$ is the infinitesimal generator of a $C_0-$semigroup on  $X$.
            Such operators are, for example the Laplacian $\Delta$ on
            $X=W^{2,2}(\R^n)$ or any strictly elliptic linear partial differential operator  of even order
            $P(x,D)=\sum_{|\iota|\leq 2m}a_\iota(x)D^\iota$. For example, second order elliptic operators can be written in
            divergence form $L=\nabla\cdot(Q\nabla\cdot
            +b)+c\nabla\cdot$, where $Q$ is a positively definite function matrix. 
            \item[2)] $A_\alpha=A+R_\alpha, \alpha\in\mathcal I$, where $A$ is as in 1), while  $R_\alpha, \alpha\in \mathcal{I},$ are bounded linear operators on $X$
            so that $\mathbf R$
            is a coordinatewise operator
            $${\mathbf R}U(t,\omega)=\sum_{\alpha\in \mathcal{I}}R_\alpha
            u_\alpha(t)H_\alpha(\omega).$$

            Especially, if we take $A=0$ and $R_\alpha$ to be
            multiplication operators $R_\alpha(x)=r_\alpha\cdot x$,
            $x\in X$, then the resulting operator $\mathbf R$ is a
            self-adjoint operator with eigenvalues $r_\alpha$
            corresponding to the eigenvectors $H_\alpha$ and thus
            represents a natural generalization of the
            Ornstein-Uhlenbeck operator. For $r_\alpha=|\alpha|$, $\alpha\in\mathcal I$, we
            retrieve the Ornstein-Uhlenbeck operator $\mathcal R$.

            Finally, we note that every bounded linear coordinatewise operator
            $\mathbf R$ can be written in the form
            ${\mathbf R}u=\delta(\mathbf Mu)$ where $\mathbf M$ is a generalization
            of the Malliavin derivative. This will be done in
            Proposition \ref{Dora1}.


            \end{enumerate}

\item[II)]  Special cases for $\mathbf B$:
            \begin{enumerate}
            \item[1)] $\mathbf B$ is an operator acting as a multiplication operator with a deterministic
            function, i.e. $B_\alpha=b$ for $\alpha=(0,0,0,0,\ldots)$ and $B_\alpha=0$ for all other  $\alpha\in\mathcal I$. Its
            action is thus
            $$\mathbf B\lozenge U(t,\omega)=\sum_{\alpha\in \mathcal{I}}b\cdot u_\alpha(t)H_\alpha(\omega).$$
            For example, we may take $X=L^2(\mathbb R^n)$ and $b=b(x)$, $x\in\mathbb R^n$, for an essentially bounded function $b$.

            \item[2)] $\mathbf B$ is multiplication with spatial white
            noise on $X=L^2(\mathbb R^n)$. Let $B_k:= B_{\varepsilon_k} = \xi_k$, $k\in\N$, and
            $B_\alpha=0$ for $\alpha\neq\varepsilon_k$, i.e.
            $B_k(v(x))=\xi_k(x)\cdot v(x)$, $k\in\N$.
            Then, $$\mathbf B\lozenge U(t,\omega)  =
            W(x,\omega)\lozenge U(t,\omega)  .$$
            Clearly, \begin{eqnarray*}\mathbf B\lozenge U &=& \sum_{\gamma\in\mathcal
            I}\sum_{k\in\N} B_k(u_{\alpha-\varepsilon_k})H_\gamma = \sum_{\gamma\in\mathcal
            I}\sum_{k\in\N} u_{\alpha-\varepsilon_k}\xi_k H_\gamma\\ &
            =&
            \sum_{\gamma\in\mathcal
            I}\sum_{\alpha+\varepsilon_k=\gamma} u_{\alpha}\xi_k
            H_\gamma = W\lozenge U.\end{eqnarray*}
            Multiplication with spatial white noise is important for
            applications since it describes stationary
            perturbations.

 \item[3)]  $\mathbf B$ is of the form $B_{\varepsilon_k}=B_k$, $k\in\N$, and
            $B_\alpha=0$ for $\alpha\neq\varepsilon_k$, where
            $B_k:X\rightarrow X$, $k\in\N$, are bounded linear operators.

             Note that in this case there is a one-to-one correspondence between operators of
 the form $\mathbf B\lozenge$ and operators of the form $\delta(\mathbf Mu)$ where $\mathbf M$ is a simple coordinatewise operator. This will be done in
            Proposition \ref{Dora2}.

 \item[4)]  $\mathbf B$ is a simple coordinatewise operator, i.e. $B_\alpha=B, \alpha\in\mathcal I,$ where $B$ is a bounded linear operator on
 $X$. Alternatively, we may also regard operators as $B:X\to X'$ in
 order to make them bounded;
            such operators are for example the divergence $\nabla\cdot$
            as a mapping from
            $X=W^{1,2}(\R^n)$ to $X'=W^{-1,2}(\R^n)$.

   \item[5)] $\mathbf B\lozenge
   =\nabla\cdot\lozenge(Q\lozenge\nabla\cdot
            +b\lozenge)+c\lozenge\nabla\cdot$ as a strictly elliptic second order operator with random coefficients. This operator is obtained
            for $B_\alpha= \nabla\cdot(Q_\alpha\nabla\cdot
            +b_\alpha)+c_\alpha\nabla\cdot$, $\alpha\in\mathcal I$, and was studied in \cite{ps} and
            \cite{DP2}.

\end{enumerate}
\end{enumerate}

\begin{proposition}\label{Dora1}
Let $\mathbf R:X\otimes (S)_{-1}\rightarrow X\otimes(S)_{-1}$ be a
bounded linear coordinatewise operator defined by ${\mathbf
R}u(t,\omega)=\sum_{\alpha\in \mathcal{I}}R_\alpha
            u_\alpha(t)H_\alpha(\omega)$.
\begin{enumerate}
\item
            There exists an operator
            $\mathbf M:X\otimes (S)_{-1}\rightarrow
            X\otimes\Svarcprim\otimes(S)_{-1}$ of the form
            $$\mathbf Mu=\sum_{k=1}^\infty \mathbf M_ku\otimes\xi_k,\quad u\in X\otimes (S)_{-1},$$
            for some coordinatewise operators $\mathbf M_k:X\otimes (S)_{-1}\rightarrow
            X\otimes(S)_{-1}$, $k\in\N$, such that
            $${\mathbf R}u=\delta(\mathbf Mu)$$ holds.
\item
            Especially, if $\mathbf R$ is a selfadjoint operator,
            then $\mathbf M$ is a generalization of the Malliavin
            derivative.
\end{enumerate}
\end{proposition}

\begin{proof}
a)  In \cite{AADM} we proved that the Skorokhod integral is
invertible, i.e. there exists a unique solution to equations of the
form $\delta(v)=f$. Considering the equation $\delta(\mathbf
Mu)=\sum_{\alpha\in \mathcal{I}}R_\alpha
            u_\alpha\;H_\alpha$ and applying the result from
            \cite{AADM}, we obtain $\mathbf Mu$ in the form
$$\mathbf Mu=\sum_{\alpha\in\mathcal I}
\sum_{k\in\N}(\alpha_k+1)\frac{R_{\alpha+\varepsilon_k}(
u_{\alpha+\varepsilon_k})}{|\alpha+\varepsilon_k|}\otimes\xi_k\otimes
H_\alpha.$$ By defining $$\mathbf M_ku=\sum_{\alpha\in\mathcal I}
(\alpha_k+1)\frac{R_{\alpha+\varepsilon_k}(
u_{\alpha+\varepsilon_k})}{|\alpha+\varepsilon_k|}\otimes
H_\alpha,\quad k\in\N,$$ we obtain the assertion.

 b)                Let $\mathbf R$ be a self-adjoint operator with eigenvalues $r_\alpha$ and
                eigenvectors $H_\alpha$, $\alpha\in\mathcal I$, i.e., an operator of
                the form ${\mathbf R}u=\sum_{\alpha\in\mathcal I} r_\alpha u_\alpha
                H_\alpha$. Assume that $r_\alpha=\sum_{k\in\mathbb
                N}r_{k,\alpha}$ for some $r_{k,\alpha}\in \R$, $k\in\mathbb N$, $\alpha\in\mathcal
                I$, is an arbitrary decomposition of the value
                $r_\alpha$. 

                Define $$\mathbf M_k u=\sum_{\alpha\in\mathcal
                I}r_{k,\alpha}u_\alpha\otimes
                H_{\alpha-\varepsilon_k}.$$ Then $\mathbf Mu=\sum_{k\in\mathbb N}
                \mathbf M_ku\otimes \xi_k = \sum_{k\in\mathbb N}\sum_{\alpha\in\mathcal
                I}r_{k,\alpha}u_\alpha \otimes H_{\alpha-\varepsilon_k}\otimes\xi_k $ and
                $$\delta(\mathbf Mu)=\sum_{k\in\mathbb N}\sum_{\alpha\in\mathcal
                I}r_{k,\alpha}u_\alpha\otimes H_{\alpha}=\sum_{\alpha\in\mathcal I}r_\alpha
                u_\alpha\otimes H_\alpha.$$
\end{proof}

\begin{remark}
The converse is not true. Even if each $\mathbf M_k$, $k\in\N$, is a
simple coordinatewise operator (and so is $\mathbf M$), ${\mathbf R}:=\delta\circ
\mathbf M$ does not need to be a coordinatewise operator. This would
require that the system $R_\alpha(u_\alpha)=\sum_{k\in\N}
m_k(u_{\alpha-\varepsilon_k})$, $\alpha\in\mathcal I$, is solvable for $R_\alpha(\cdot)$
given the functions $m_k(\cdot)$, $k\in\N$, which is not true in general.
\end{remark}

\begin{proposition}\label{Dora2}
Let $\mathbf M:X\otimes (S)_{-1}\rightarrow
            X\otimes\Svarcprim\otimes(S)_{-1}$ be of the form
            \begin{equation}\mathbf Mu=\sum_{k=1}^\infty \mathbf M_ku\otimes\xi_k,\quad u\in X\otimes
            (S)_{-1},\label{LR}\end{equation}
            for some simple coordinatewise operators $\mathbf M_k:X\otimes (S)_{-1}\rightarrow
            X\otimes(S)_{-1}$, $k\in\N$.  Then, there exists a coordinatewise
            operator
            $\mathbf B$ such that
            $B_\alpha=0$ for $\alpha\neq\varepsilon_k$, $k\in\N$, and
            $$\delta(\mathbf Mu)={\mathbf B}\lozenge u$$ holds.

            Conversely, for any coordinatewise operator
            $\mathbf B$ such that
            $B_\alpha=0$ for $\alpha\neq\varepsilon_k$, $k\in\N$,
            there exists an operator $\mathbf M$ of the form $\mathbf Mu=\sum_{k=1}^\infty
            \mathbf M_ku\otimes\xi_k$ for some simple coordinatewise operators $\mathbf M_k$,
            $k\in\N$, such that $\delta(\mathbf Mu)={\mathbf B}\lozenge u$ holds.

\end{proposition}

\begin{proof}
Let $\mathbf M$ be an operator as stated above and since $\mathbf
M_k$ are simple coordinatewise operators, we can write them as
$$\mathbf M_k(u)=\sum_{\alpha\in\mathcal I} m_k(u_\alpha)H_\alpha,\qquad
u=\sum_{\alpha\in\mathcal I}u_\alpha H_\alpha,$$ for some operators
$m_k:X\rightarrow X$, $k\in\N$. Thus,
$$\mathbf Mu=\sum_{k=1}^\infty \sum_{\alpha\in\mathcal I}
m_k(u_\alpha)H_\alpha \otimes \xi_k$$ which further implies
\begin{equation}\delta(\mathbf Mu) = \sum_{k=1}^\infty \sum_{\alpha\in\mathcal I}
m_k(u_\alpha)H_{\alpha+\varepsilon_k}  = \sum_{k=1}^\infty
\sum_{\alpha\in\mathcal I}
m_k(u_{\alpha-\varepsilon_k})H_{\alpha}.\label{SP1}\end{equation}

On the other hand, if $\mathbf B$ is such that
            $B_\alpha=0$ for $\alpha\neq\varepsilon_k$, $k\in\N$,
            and we denote by $B_k:=B_{\varepsilon_k}$, $k\in\mathbb
            N$, the operators acting on $X$, then
\begin{equation}
\mathbf B \lozenge u = \sum_{\alpha\in\mathcal I}\sum_{k=1}^\infty
B_k(u_{\alpha-\varepsilon_k})H_{\alpha}.
 \label{SP2}
\end{equation}
From \eqref{SP1} and \eqref{SP2} it follows that $\delta(\mathbf Mu)
= \mathbf B \lozenge u$ if and only if $m_k=B_k$ for all $k\in\N$.
Thus, there is a one-to-one correspondence between the operators
$\mathbf B\lozenge$ and $\delta\circ \mathbf M$.
\end{proof}

\begin{remark}
In \cite{Boris} and \cite{Boris2} Rozovskii and Lototsky considered
the equation $\frac{d}{dt} = \mathbf A u + \delta(\mathbf Mu)+f$,
where $\mathbf M$ is of the form \eqref{LR}. They implicitly assumed
that all their operators $\mathbf A$ and $\mathbf M_k$, $k\in\N$,
belong to our class of simple coordinatewise operators. This corresponds to our special cases I-1) and
II-3).
\end{remark}

Some special cases of stochastic differential equations covered by
\eqref{SPDE-1} include the following:

\begin{itemize}
\item[$\bullet $]  The heat equation with random potential $$\frac{d}{dt}u=\Delta
u + {\mathbf B}\lozenge u.$$ In particular, if the random potential
is modeled by stationary perturbations, we may take spatial white
noise as a model and obtain $\frac{d}{dt}u=\Delta u + W\lozenge u$.
This corresponds to the special choice of operators I-1) and
II-2).

\item[$\bullet $]  The heat equation in random (inhomogeneous and anisotropic)
media, where the physical properties of the medium are modeled by a
stochastic matrix $Q$. This corresponds to the case I-1) with
$\mathbf A=0$ and II-5) leading to an equation of the form
$$\frac{d}{dt}u=\nabla\cdot\lozenge(Q\lozenge\nabla\cdot u
            +b\lozenge u)+c\lozenge\nabla\cdot u +f.$$

\item[$\bullet $]  Taking $\mathbf A=0$ and $B_k:=B_{\varepsilon_k}=\xi_k \nabla\cdot$, $k\in\N$, (see special cases II-2) and II-4)) we
obtain the transport equation driven by white noise
$$\frac{d}{dt}u=\Delta u + W\lozenge \nabla\cdot u.$$ 

\item[$\bullet $]
The Langevin equation $$\frac{d}{dt}u= -\lambda u+W(t),$$
$\lambda>0$, corresponding to the case I-1) with $A=-\lambda$,
$f=W$ and $\mathbf B=0$. Its solution is the Ornstein-Uhlenbeck
process describing the spatial position of a Brownian particle in a
fluid with viscosity $\lambda$.

In \cite{Dejv} the authors considered the generalized Langevin equation leading to generalized Ornstein-Uhlenbeck operators driven by L\'evy processes
$$ \frac{d}{dt}u= Ju+C(\frac{d}{dt}Y),$$ where $Y$ is a L\'evy process, $J$ the infinitesimal generator of a $C_0-$semigroup and $C$ a bounded operator. All processes are Hilbert space valued. This corresponds to our case with $X$ being this Hilbert space, $\mathbf A=J$, $\mathbf B=0$ and $f=C(Y')$.

\item[$\bullet $]  The equation $\frac{d}{dt} = \mathbf A u + \delta(\mathbf Mu)+f$, that
was extensively studied in \cite{Boris} and \cite{Boris2}. This
corresponds to our special cases I-1) and II-3).

\item[$\bullet $]
The equation $$\frac{d}{dt} u= L u + W \lozenge u,$$ where $L$ is a
strictly elliptic partial differential operator as studied in
\cite{Cat2} and \cite{yuhu}. This corresponds to the special case
I-1) and II-2).


\end{itemize}


\section{Stochastic evolution equations}\label{1.1}

Now we turn to the general case of stochastic Cauchy problems of the form $\frac{d}{dt}U(t,\omega) = {\mathbf A}U(t,\omega) + {\mathbf B}\lozenge U(t,\omega) +
F(t,\omega),$ $t\in(0,T]$, $\omega\in\Omega$, with initial value $U(0,\omega)=U^0(\omega)$, $\omega\in\Omega$, and all processes are $X$-valued for a Banach space $X$.

\begin{definition}\label{U resaca CP}
It is said that $U$ is a solution to (\ref{nasaJNA}) if $U\in C([0,T], X)\otimes (S)_{-1}\cap C^1((0,T], X)\otimes (S)_{-1}$ and $U$ satisfies (\ref{nasaJNA}).
\end{definition}

\begin{theorem}\label{th polugr} Let  $\mathbf{A}$ be a coordinatewise operator of the form (\ref{delovanje A}), where the operators $A_\alpha,\;\alpha\in \mathcal{I},$ defined on the same domain $D$ dense in $X,$ are infinitesimal generators of $C_0-$semigroups $(T_t)_\alpha,\;t\geq 0,\;\alpha\in \mathcal{I},$  uniformly bounded by
\begin{equation}\label{A1}
\|(T_t)_\alpha\|_{L(X)}\leq M e^{w t},\;t\geq 0,\quad \mbox{ for some }\;M, w>0.
\end{equation}
Let $\mathbf{B}\lozenge$ be of the form (\ref{delovanje B}), where $B_\alpha,\;\alpha\in \mathcal{I},$ are bounded linear operators on $X$  so that there exists $p>0$ such that
\begin{equation}\label{ogr B}
K:=\sum_{\alpha\in \mathcal{I}}\|B_\alpha\|(2\N)^{-p\frac{\alpha}{2}}<\infty.
\end{equation}
Let the initial value $U^0\in X\otimes (S)_{-1}$ be such that $U^0\in Dom(\mathbf{A})$ i.e.
  \begin{equation}\label{A2}
\begin{split}
 U^0(\omega)=\sum_{\alpha\in \mathcal{I}}u^0_\alpha H_\alpha(\omega)\in X\otimes (S)_{-1,-p},\;\mbox{satisfies}\quad\sum_{\alpha\in \mathcal{I}}\|u_\alpha^0\|_X^2 (2\N)^{-p \alpha}<\infty;
 \end{split}
 \end{equation}
 and
   \begin{equation}\label{A4}
\begin{split}
 \mathbf{A}U^0(\omega)=\sum_{\alpha\in \mathcal{I}}A_\alpha u^0_\alpha H_\alpha(\omega)\in X\otimes (S)_{-1,-p},\;\mbox{satisfies}\quad\sum_{\alpha\in \mathcal{I}}\|A_\alpha u_\alpha^0\|_X^2 (2\N)^{-p \alpha}<\infty.
 \end{split}
 \end{equation}
Moreover, let
\begin{equation}\label{A5}\begin{split} &F(t,\omega)=\sum_{\alpha\in \mathcal{I}}f_\alpha(t) H_\alpha(\omega)\in C^1([0,T],X)\otimes (S)_{-1},\;\;\;t\mapsto f_\alpha(t)\in C^1([0,T],X),\;\alpha\in \mathcal{I},\\
&\mbox{so that}\; \sum_{\alpha\in \mathcal{I}}\|f_\alpha\|_{C^1([0,T],X)}^2 (2\N)^{-p \alpha}=\sum_{\alpha\in \mathcal{I}}\Big(\sup_{t\in[0,T]}\|f_\alpha(t)\|_X+\sup_{t\in[0,T]}\|f'_\alpha(t)\|_X\Big)^2 (2\N)^{-p \alpha}<\infty.\end{split}\end{equation}

Then, the stochastic Cauchy problem (\ref{nasaJNA}) has a unique solution $U$  in $C^1([0,T],X)\otimes (S)_{-1,-p}$.

\end{theorem}
\begin{proof}
We seek for the solution in form of $U(t,\omega)=\sum_{\alpha\in \mathcal{I}}u_\alpha(t)H_\alpha(\omega)$. Then, the Cauchy problem (\ref{nasaJNA}) is equivalent to the infinite system:
\begin{equation}\label{CP2}\begin{split} \frac{d}{dt}\;u_\alpha(t)&= A_\alpha u_\alpha(t)+ \sum_{\beta\leq \alpha}B_\beta u_{\alpha-\beta}(t)+ f_\alpha(t),\quad t\in (0,T] ,\\
 u_\alpha(0)&=u_\alpha^0\in D,\quad \alpha\in \mathcal{I}.  \end{split}\end{equation}

Let $\mathbf{0}$ be the multi-index $\mathbf{0}=(0,0,...)$. We rewrite (\ref{CP2}) as

\begin{equation}\label{CP3}\begin{split} \frac{d}{dt}\;u_\alpha(t)&= (A_\alpha + B_\mathbf{0}) u_\alpha(t)+  \sum_{\mathbf{0}<\beta\leq \alpha}B_\beta u_{\alpha-\beta}(t)+ f_\alpha(t),\quad t\in(0,T],\\
 u_\alpha(0)&=u_\alpha^0\in D,\quad \alpha\in \mathcal{I}.  \end{split}\end{equation}

Next, $A_\alpha +B_\mathbf{0}$ are infinitesimal generators of $C_0-$semigroups $(S_t)_\alpha$ in $X$ such that
\begin{equation}\label{granica za St}
\|(S_t)_\alpha\|\leq M e^{(w+ M\|B_\mathbf{0}\|)t},\;t\geq 0,\;\alpha\in \mathcal{I}.\end{equation}

According to Subsection \ref{0.1}, if $f_\alpha,\;\alpha\in \mathcal{I},$ fulfills condition (i) or (ii), the inhomogeneous initial value problem (\ref{CP3}) has a solution $u_\alpha(t) \in C([0,T],X)\cap C^1((0,T],X)$, $\alpha\in \mathcal{I}$, given by

\begin{equation}\label{S1}\begin{split} u_\mathbf{0}(t)&= (S_t)_\mathbf{0} u_\mathbf{0}^0+ \int_0^t (S_{t-s})_\mathbf{0} f_\mathbf{0}(s)ds,\quad t\in[0,T] \\
u_\alpha(t)&= (S_t)_\alpha u_\alpha^0+  \int_0^t (S_{t-s})_\alpha \Big(\sum_{\mathbf{0}<\beta\leq \alpha}B_\beta u_{\alpha-\beta}(s)+ f_\alpha(s)\Big)ds,\quad t\in[0,T].\\
\end{split}\end{equation}
Since $f_\alpha\in C^1([0,T],X)$ it follows by induction on $\alpha$ that
 $$\sum_{\mathbf{0}<\beta\leq \alpha}B_\beta u_{\alpha-\beta}(s)+ f_\alpha(s)\in C^1([0,T],X),\quad\mbox{ for all }\quad \alpha\in\mathcal I.$$
Thus, $u_\alpha\in C^1([0,T],X)$ and $\frac{d}{dt}u_\alpha(0)=(A_\alpha+B_{\mathbf{0}}) u^0_\alpha+\sum_{\mathbf{0}<\beta\leq \alpha}B_\beta u_{\alpha-\beta}^0+ f_\alpha(0),\;\alpha\in \mathcal{I}$.

Note that for each fixed $\alpha\in\mathcal I$, $u_\alpha(t)$ exists for all $t\in[0,T]$ and it is the unique (classical) solution on the whole interval $[0,T]$. It remains to prove that $\sum_{\alpha\in\mathcal I}u_\alpha(t)H_\alpha(\omega)$ converges in $C^1([0,T],X)\otimes (S)_{-1,-p}$.

First, we show that $U(t,\omega)=\sum_{\alpha\in \mathcal{I}}u_\alpha(t)H_\alpha(\omega)\in C^1([0,T_0],X)\otimes S_{-1,-p}$ for appropriate $T_0\in(0,T]$, i.e. we show that
\begin{equation}\label{zapis ocene T0}\sum_{\alpha\in \mathcal{I}}\|u_\alpha\|_{C^1([0,T_0],X)}^2 (2\N)^{-p \alpha}=\sum_{\alpha\in \mathcal{I}}\Big(\sup_{t\in[0,T_0]}\|u_\alpha(t)\|_X+\sup_{t\in[0,T_0]}\|\frac{d}{dt}u_\alpha(t)\|_X\Big)^2 (2\N)^{-p \alpha}<\infty.\end{equation}

Later on we will prove that the same holds if we take in (\ref{zapis ocene T0}) supremums over the intervals $[T_0,2T_0]$, $[2T_0,3T_0] ,...$ etc. Since $[0,T]$ can be covered by finitely many intervals of the form $[kT_0,(k+1)T_0],\;k\in \N_0,$ we conclude that
\begin{equation}\label{zapis ocene T}\sum_{\alpha\in \mathcal{I}}\|u_\alpha\|_{C^1([0,T],X)}^2 (2\N)^{-p \alpha}=\sum_{\alpha\in \mathcal{I}}\Big(\sup_{t\in[0,T]}\|u_\alpha(t)\|_X+\sup_{t\in[0,T]}\|\frac{d}{dt}u_\alpha(t)\|_X\Big)^2 (2\N)^{-p \alpha}<\infty.\end{equation}
In order to do this, we introduce a notation for subsets of multi-indices $$\mathcal{I}_{n,m}=\{\alpha\in \mathcal{I}:\;|\alpha|\leq n \wedge {\rm Index}(\alpha)\leq m\},\;n,m\in\N,$$ where, for $\alpha=(\alpha_1,\alpha_2,\dots,\alpha_m,0,0,\dots)\in \mathcal{I}$, we have $|\alpha|=\alpha_1+\alpha_2+\dots+\alpha_m$ and ${\rm Index}(\alpha)$ is last coordinate where $\alpha$ has a nonzero entry. For later reference, we introduce the function
\begin{equation}\label{Ct} C(t)=\frac{M^2}{(w+M\|B_\mathbf{0}\|)^2}(e^{(w+M\|B_\mathbf{0}\|)t}-1)^2\end{equation}
and fix $T_0\in(0,T]$ such that $C(T_0)<\frac{1}{5K^2}$.

First, we show that $$\sum_{\alpha\in \mathcal{I}}\|u_\alpha(t)\|_{C([0,T_0],X)}^2(2\N)^{-p\alpha}=\sum_{\alpha\in \mathcal{I}}\sup_{t\in[0,T_0]}\|u_\alpha(t)\|_X^2(2\N)^{-p\alpha}<\infty,$$ by proving that partial sums $\sum_{\alpha\in \mathcal{I}_{n,m}}\sup_{t\in[0,T_0]}\|u_\alpha(t)\|_X^2(2\N)^{-p\alpha},\;n,m\in\N,$ are bounded from above.

Using (\ref{S1}) we obtain
\begin{align*}\frac{1}{3}\sum_{\alpha\in \mathcal{I}_{n,m}}\|u_\alpha(t)\|_X^2(2\N)^{-p\alpha}&\leq\sum_{\alpha\in \mathcal{I}_{n,m}}\|(S_t)_\alpha\|^2\|u_\alpha^0\|_X^2(2\N)^{-p\alpha}\\
&+\sum_{\alpha\in \mathcal{I}_{n,m}}\Big[\int_0^t \|(S_{t-s})_\alpha\|
\sum_{\mathbf{0}<\beta\leq \alpha}\|B_\beta u_{\alpha-\beta}(s)\|_X ds\Big]^2(2\N)^{-p\alpha}\\
&+\sum_{\alpha\in \mathcal{I}_{n,m}}\Big[\int_0^t \|(S_{t-s})_\alpha\| \|f_\alpha(s)\|_X ds\Big]^2(2\N)^{-p\alpha}.
\end{align*}
The first term on the right-hand side, for all $t\in[0,T_0]$, having in mind (\ref{A2}) and (\ref{granica za St}), satisfies
\begin{align}\label{Q1}
\sum_{\alpha\in \mathcal{I}_{n,m}}\|(S_t)_\alpha\|^2\|u_\alpha^0\|_X^2(2\N)^{-p\alpha}&\leq \sum_{\alpha\in \mathcal{I}}\|(S_t)_\alpha\|^2\|u_\alpha^0\|_X^2(2\N)^{-p\alpha}\nonumber\\
&\leq M^2e^{2(w+M\|B_\mathbf{0}\|)T_0}\sum_{\alpha\in \mathcal{I}}\|u_\alpha^0\|_X^2(2\N)^{-p\alpha}:=Q_1<\infty.
\end{align}
Similarly, for all $t\in[0,T_0]$, using (\ref{A5}) and (\ref{granica za St}), the third term satisfies
\begin{align}\label{G}
\sum_{\alpha\in \mathcal{I}_{n,m}}&\Big[\int_0^t \|(S_{t-s})_\alpha\| \|f_\alpha(s)\|_X ds\Big]^2(2\N)^{-p\alpha}\leq \sum_{\alpha\in \mathcal{I}}\Big[\int_0^t \|(S_{t-s})_\alpha\| \|f_\alpha(s)\|_X ds\Big]^2(2\N)^{-p\alpha}\nonumber\\
&\leq \Big[\int_0^t Me^{(w+M\|B_\mathbf{0}\|)(t-s)}ds\Big]^2 \sum_{\alpha\in \mathcal{I}}\sup_{s\in[0,t]} \|f_\alpha(s)\|_X^2(2\N)^{-p\alpha}\nonumber\\
&\leq \frac{M^2}{(w+M\|B_\mathbf{0}\|)^2}\Big(e^{(w+M\|B_\mathbf{0}\|)T_0}-1\Big)^2\sum_{\alpha\in \mathcal{I}}\sup_{t\in[0,T]} \|f_\alpha(t)\|_X^2(2\N)^{-p\alpha}:=G<\infty.
\end{align}
Note  that  in (\ref{G}) we took the supremum over the whole interval $[0,T]$.

For the second term, using (\ref{ogr B}), (\ref{granica za St}), (\ref{Ct}) and the generalized Minkowski inequality, we obtain
\begin{align}\label{E}
\sum_{\alpha\in \mathcal{I}_{n,m}}&\Big[\int_0^t \|(S_{t-s})_\alpha\|
\sum_{\beta+\gamma=\alpha}\|B_\beta\|\| u_{\gamma}(s)\|_X ds\Big]^2(2\N)^{-p\alpha}\nonumber\\
&\leq\frac{M^2}{(w+M\|B_\mathbf{0}\|)^2}\Big(e^{(w+M\|B_\mathbf{0}\|)t}-1\Big)^2\sum_{\alpha\in \mathcal{I}_{n,m}} \Big
[\sum_{\beta+\gamma=\alpha}\sup_{s\in[0,t]}\|B_\beta\|\| u_{\gamma}(s)\|_X\Big]^2(2\N)^{-p\alpha}\nonumber\\
&\leq C(T_0)\Big(\sum_{\beta\in \mathcal{I}_{n,m}}\|B_\beta\|(2\N)^{-p\frac{\beta}{2}}\Big)^2\Big(\sum_{\gamma\in \mathcal{I}_{n,m}}\sup_{t\in[0,T_0]}\|u_\gamma(t)\|_X^2(2\N)^{-p\gamma}\Big)
\nonumber\\
&\leq C(T_0)K^2\sum_{\alpha\in \mathcal{I}_{n,m}}\sup_{t\in[0,T_0]}\|u_\alpha(t)\|_X^2(2\N)^{-p\alpha}.
\end{align}
Finally, for all $n,m\in\N$, we obtain
\begin{align*}
\frac{1}{3}\sum_{\alpha\in \mathcal{I}_{n,m}}\sup_{t\in[0,T_0]}\|u_\alpha(t)\|_X^2(2\N)^{-p\alpha}\leq Q_1+G+C(T_0)K^2 \sum_{\alpha\in \mathcal{I}_{n,m}}\sup_{t\in[0,T_0]}\|u_\alpha(t)\|_X^2(2\N)^{-p\alpha}.
\end{align*}
Since $\frac{1}{3}-C(T_0)K^2>\frac{1}{5}-C(T_0)K^2>0$, we have
\begin{align}\label{parc suma}
\sum_{\alpha\in \mathcal{I}_{n,m}}\sup_{t\in[0,T_0]}\|u_\alpha(t)\|_X^2(2\N)^{-p\alpha}\leq \frac{Q_1+G}{\frac{1}{3}-C(T_0)K^2}.
\end{align}

Let $(m_n)_{n\in\N}$ be an arbitrary sequence of positive integers tending to infinity. Then, $$\sum_{\alpha\in \mathcal{I}}\sup_{t\in[0,T_0]}\|u_\alpha(t)\|_X^2(2\N)^{-p\alpha}=\lim_{n\to \infty}\sum_{\alpha\in \mathcal{I}_{n,m_n}}\sup_{t\in[0,T_0]}\|u_\alpha(t)\|_X^2(2\N)^{-p\alpha}\leq \frac{Q_1+G}{\frac{1}{3}-C(T_0)K^2},$$ since it is a series of positive numbers and thus does not depend on the order of summation.

Now we show that $$\sum_{\alpha\in \mathcal{I}}\|\frac{d}{dt}u_\alpha(t)\|_{C([0,T_0],X)}^2(2\N)^{-p\alpha}=\sum_{\alpha\in \mathcal{I}}\sup_{t\in[0,T_0]}\|\frac{d}{dt}u_\alpha(t)\|_X^2(2\N)^{-p\alpha}<\infty.$$ In order to acomplish that, we differentiate (\ref{S1}) with respect to $t$, and obtain
\begin{equation}\label{DS1}\begin{split} \frac{d}{dt}u_\mathbf{0}(t)&= (S_t)_\mathbf{0}(A_\mathbf{0}+B_\mathbf{0})u_\mathbf{0}^0+ \int_0^t (S_{t-s})_\mathbf{0}\frac{d}{ds} f_\mathbf{0}(s)ds+(S_t)_\mathbf{0}f(0),\quad t\in[0,T] ,\\
\frac{d}{dt}u_\alpha(t)&= (S_t)_\alpha (A_\alpha+B_\mathbf{0})u_\alpha^0+  \int_0^t (S_{t-s})_\alpha \Big(\sum_{\mathbf{0}<\beta\leq \alpha}B_\beta \frac{d}{ds} u_{\alpha-\beta}(s)+ \frac{d}{ds}f_\alpha(s)\Big)ds \\
& +(S_t)_\alpha\Big(\sum_{\mathbf{0}<\beta\leq \alpha}B_\beta u_{\alpha-\beta}(0)+f_\alpha(0)\Big) ,\quad t\in[0,T],\quad \alpha\in \mathcal{I}.\end{split}\end{equation}
In the sequel we estimate partial sums of $\sum_{\alpha\in \mathcal{I}}\sup_{t\in[0,T_0]}\|\frac{d}{dt}u_\alpha(t)\|_X^2(2\N)^{-p\alpha}$. So,

\begin{align*}\frac{1}{5}\sum_{\alpha\in \mathcal{I}_{n,m}}\|\frac{d}{dt}u_\alpha(t)\|_X^2(2\N)^{-p\alpha}&\leq\sum_{\alpha\in \mathcal{I}_{n,m}}\|(S_t)_\alpha\|^2\|(A_\alpha+B_\mathbf{0})u_\alpha^0\|_X^2(2\N)^{-p\alpha}\\
&+\sum_{\alpha\in \mathcal{I}_{n,m}}\Big[\int_0^t \|(S_{t-s})_\alpha\|
\sum_{\mathbf{0}<\beta\leq \alpha}\|B_\beta \frac{d}{ds}u_{\alpha-\beta}(s)\|_X ds\Big]^2(2\N)^{-p\alpha}\\
&+\sum_{\alpha\in \mathcal{I}_{n,m}}\Big[\int_0^t \|(S_{t-s})_\alpha\| \|\frac{d}{ds}f_\alpha(s)\|_X ds\Big]^2(2\N)^{-p\alpha}\\
&+\sum_{\alpha\in \mathcal{I}_{n,m}}\|(S_{t})_\alpha\|^2 \Big
[\sum_{\mathbf{0}<\beta\leq \alpha}\|B_\beta u_{\alpha-\beta}(0)\|_X\Big]^2(2\N)^{-p\alpha}\\
&+\sum_{\alpha\in \mathcal{I}_{n,m}} \|(S_{t})_\alpha\|^2 \|f_\alpha(0)\|_X^2 (2\N)^{-p\alpha}.
\end{align*}
According to (\ref{A2}) and (\ref{A4}), we obtain $\sum_{\alpha\in \mathcal{I}}(A_\alpha+B_\mathbf{0})u_\alpha^0H_\alpha(\omega)\in X\otimes (S)_{-1,-p} $. So the first term on the right-hand side can be evaluated by
\begin{align}\label{Q'1}
\sum_{\alpha\in \mathcal{I}_{n,m}}&\|(S_t)_\alpha\|^2\|(A_\alpha+B_\mathbf{0})u_\alpha^0\|_X^2(2\N)^{-p\alpha}\leq \sum_{\alpha\in \mathcal{I}}\|(S_t)_\alpha\|^2\|(A_\alpha+B_\mathbf{0})u_\alpha^0\|_X^2(2\N)^{-p\alpha}\nonumber\\
&\leq M^2e^{2(w+M\|B_\mathbf{0}\|)T_0}\sum_{\alpha\in \mathcal{I}}\|(A_\alpha+B_\mathbf{0})u_\alpha^0\|_X^2(2\N)^{-p\alpha}:=Q'_1<\infty.
\end{align}
The third term, for all $t\in[0,T_0]$, satisfies
\begin{align}\label{G'}
\sum_{\alpha\in \mathcal{I}_{n,m}}&\Big[\int_0^t \|(S_{t-s})_\alpha\| \|\frac{d}{ds}f_\alpha(s)\| ds\Big]^2(2\N)^{-p\alpha}\leq \sum_{\alpha\in \mathcal{I}}\Big[\int_0^t \|(S_{t-s})_\alpha\| \|\frac{d}{ds}f_\alpha(s)\|_X ds\Big]^2(2\N)^{-p\alpha}\nonumber\\
&\leq \frac{M^2}{(w+M\|B_\mathbf{0}\|)^2}\Big(e^{(w+M\|B_\mathbf{0}\|)T_0}-1\Big)^2\sum_{\alpha\in \mathcal{I}}\sup_{t\in[0,T]} \|\frac{d}{ds}f_\alpha(t)\|_X^2(2\N)^{-p\alpha}:=G'<\infty.
\end{align}
The fourth term, using (\ref{ogr B}), (\ref{A2}), (\ref{granica za St}) and the generalized Minkowski inequality, can be estimated by
\begin{align}\label{H'}
\sum_{\alpha\in \mathcal{I}_{n,m}}&\|(S_{t})_\alpha\|^2 \Big
[\sum_{\mathbf{0}<\beta\leq \alpha}\|B_\beta u_{\alpha-\beta}(0)\|_X\Big]^2(2\N)^{-p\alpha} \leq \sum_{\alpha\in \mathcal{I}}\|(S_{t})_\alpha\|^2 \Big
[\sum_{\beta+\gamma=\alpha}\|B_\beta u_{\gamma}^0\|_X\Big]^2(2\N)^{-p\alpha}\nonumber\\
&\leq M^2e^{2(w+M\|B_\mathbf{0}\|)t}\sum_{\alpha\in \mathcal{I}}\Big
[\sum_{\beta+\gamma=\alpha}\|B_\beta\|\| u_{\gamma}^0\|_X\Big]^2(2\N)^{-p\alpha}\nonumber\\
&\leq M^2e^{2(w+M\|B_\mathbf{0}\|)T_0}\Big(\sum_{\beta\in \mathcal{I}}\|B_\beta\|(2\N)^{-p\frac{\beta}{2}}\Big)^2\Big(\sum_{\gamma\in \mathcal{I}}\|u_\gamma^0\|_X^2(2\N)^{-p\gamma}\Big):=H_1'<\infty.
\end{align}
For the fifth term, using (\ref{A5}) and (\ref{granica za St}), we have
\begin{align}\label{N'}
\sum_{\alpha\in \mathcal{I}_{n,m}} &\|(S_{t})_\alpha\|^2 \|f_\alpha(0)\|_X^2 (2\N)^{-p\alpha}\leq \sum_{\alpha\in \mathcal{I}} \|(S_{t})_\alpha\|^2 \|f_\alpha(0)\|_X^2 (2\N)^{-p\alpha}\nonumber\\
&\leq M^2e^{2(w+M\|B_\mathbf{0}\|)T_0}\sum_{\alpha\in \mathcal{I}}\sup_{t\in[0,T]} \|f_\alpha(t)\|_X^2(2\N)^{-p\alpha}:=N'<\infty.
\end{align}
Finally, for the second term, using (\ref{ogr B}), (\ref{granica za St}), (\ref{Ct}) and the generalized Minkowski inequality, we obtain
\begin{align}\label{E'}
\sum_{\alpha\in \mathcal{I}_{n,m}}&\Big[\int_0^t \|(S_{t-s})_\alpha\|
\sum_{\beta+\gamma=\alpha}\|B_\beta\|\|\frac{d}{ds} u_{\gamma}(s)\|_X ds\Big]^2(2\N)^{-p\alpha}\nonumber\\
&\leq \frac{M^2}{(w+M\|B_\mathbf{0}\|)^2}\Big(e^{(w+M\|B_\mathbf{0}\|)t}-1\Big)^2\sum_{\alpha\in \mathcal{I}_{n,m}} \Big
[\sum_{\beta+\gamma=\alpha}\sup_{s\in[0,t]}\|B_\beta\|\|\frac{d}{ds} u_{\gamma}(s)\|_X\Big]^2(2\N)^{-p\alpha}\nonumber\\
&\leq C(t)\Big(\sum_{\beta\in \mathcal{I}_{n,m}}\|B_\beta\|(2\N)^{-p\frac{\beta}{2}}\Big)^2\Big(\sum_{\gamma\in \mathcal{I}_{n,m}}\sup_{s\in[0,t]}\|\frac{d}{dt}u_\gamma(s)\|_X^2(2\N)^{-p\gamma}\Big)
\nonumber\\
&\leq C(T_0)K^2\sum_{\alpha\in \mathcal{I}_{n,m}}\sup_{t\in[0,T_0]}\|\frac{d}{dt}u_\alpha(t)\|_X^2(2\N)^{-p\alpha}.
\end{align}
 Finally, for all $n,m\in\N$, we obtain
\begin{align*}
\frac{1}{5}\sum_{\alpha\in \mathcal{I}_{n,m}}\sup_{t\in[0,T_0]}\|\frac{d}{dt}u_\alpha(t)\|_X^2(2\N)^{-p\alpha}\leq & Q'_1+G'+H'_1+N'\\
&+C(T_0)K^2 \sum_{\alpha\in \mathcal{I}_{n,m}}\sup_{t\in[0,T_0]}\|\frac{d}{dt}u_\alpha(t)\|_X^2(2\N)^{-p\alpha}.
\end{align*}
Since $\frac{1}{5}-C(T_0)K^2>0$, we have
\begin{align}\label{parc suma'}
\sum_{\alpha\in \mathcal{I}_{n,m}}\sup_{t\in[0,T_0]}\|\frac{d}{dt}u_\alpha(t)\|_X^2(2\N)^{-p\alpha}\leq \frac{Q'_1+G'+H'_1+N'}{\frac{1}{5}-C(T_0)K^2}.
\end{align}
Again, taking $(m_n)_{n\in\N}$ to be an arbitrary sequence of positive integers tending to infinity, we have $$\sum_{\alpha\in \mathcal{I}}\sup_{t\in[0,T_0]}\|\frac{d}{dt}u_\alpha(t)\|_X^2(2\N)^{-p\alpha}=\lim_{n\to \infty}\sum_{\alpha\in \mathcal{I}_{n,m_n}}\sup_{t\in[0,T_0]}\|\frac{d}{dt}u_\alpha(t)\|_X^2(2\N)^{-p\alpha}\leq \frac{Q'_1+G'+H'_1+N'}{\frac{1}{5}-C(T_0)K^2}.$$
Therefore, we obtain
\begin{equation}\label{ocena na T0}\begin{split} &U(t,\omega)\in C^1([0,T_0],X)\otimes (S)_{-1,-p},\;\mbox{i.e.}\\
&\sum_{\alpha\in \mathcal{I}}\Big(\sup_{t\in[0,T_0]}\|u_\alpha(t)\|_X+\sup_{t\in[0,T_0]}\|\frac{d}{dt}u_\alpha(t)\|_X\Big)^2 (2\N)^{-p \alpha}\leq\\
&2\sum_{\alpha\in \mathcal{I}}\Big(\sup_{t\in[0,T_0]}\|u_\alpha(t)\|^2_X+\sup_{t\in[0,T_0]}\|\frac{d}{dt}u_\alpha(t)\|^2_X\Big) (2\N)^{-p \alpha}<\infty.\end{split}
\end{equation}

Next, we consider in (\ref{ocena na T0}) supremums over the interval $[T_0,2T_0]$. On  $[T_0,2T_0]$ one can rewrite the initial value problem (\ref{CP2}) in the following equivalent form:
\begin{equation}\label{CP4}\begin{split} \frac{d}{dt}\;v_\alpha(t)&= A_\alpha v_\alpha(t)+ \sum_{\beta\leq \alpha}B_\beta v_{\alpha-\beta}(t)+ f_\alpha(T_0+t),\quad t\in (0,T_0] \\
 v_\alpha(0)&=v_\alpha^0:=u_\alpha(T_0),\quad \alpha\in \mathcal{I}.  \end{split}\end{equation}
The semigroup corresponding to the generator $A_\alpha+B_\mathbf{0}$ in (\ref{CP4}) is again the semigroup $(S_t)_\alpha,\;t\geq 0$. Using (\ref{CP2}) and (\ref{ocena na T0}), we have that $U(t,\omega)\in Dom(\mathbf{A}),$ for all $t\in[0,T_0],$ and $\mathbf{A}U(t,\omega)\in X\otimes(S)_{-1,-p},\;t\in[0,T_0].$ According to this we have that $V^0(\omega)=U(T_0,\omega)=\sum_{\alpha\in \mathcal{I}}v^0_\alpha H_\alpha(\omega)\in Dom(\mathbf{A})$ and $\mathbf{A}V^0(\omega)\in X\otimes(S)_{-1,-p}.$ Thus,
$$v_\alpha(t)= (S_t)_\alpha v_\alpha^0+  \int_0^t (S_{t-s})_\alpha \Big(\sum_{\mathbf{0}<\beta\leq \alpha}B_\beta v_{\alpha-\beta}(s)+ f_\alpha(T_0+s)\Big)ds,\quad t\in[0,T_0],$$
and clearly $v_\alpha(t)=u_\alpha(T_0+t),\;t\in[0,T_0],\;\alpha\in \mathcal{I}$.

When approximating partial sums of $\sum_{\alpha\in \mathcal{I}}\sup_{t\in[0,T_0]}\|v_\alpha(t)\|_X^2(2\N)^{-p\alpha}$, comparing to the previous calculations for $u_\alpha(t)$, only the constant $Q_1$ will be different, and here, we denote it by $Q_2$, so we again obtain
$$\sum_{\alpha\in \mathcal{I}}\sup_{t\in[0,T_0]}\|v_\alpha(t)\|_X^2(2\N)^{-p\alpha}=\sum_{\alpha\in \mathcal{I}}\sup_{t\in[T_0,2T_0]}\|u_\alpha(t)\|_X^2(2\N)^{-p\alpha}\leq \frac{Q_2+G}{\frac{1}{3}-C(T_0)K^2}.$$
Similarly, for the derivative $\frac{d}{dt}V(t,\omega)$ we obtain
$$\sum_{\alpha\in \mathcal{I}}\sup_{t\in[0,T_0]}\|\frac{d}{dt}v_\alpha(t)\|_X^2(2\N)^{-p\alpha} \leq \frac{Q'_2+G'+H'_2+N'}{\frac{1}{5}-C(T_0)K^2},$$
where, comparing to the estimates of $\frac{d}{dt}U(t,\omega)$, only the constants $Q'_1$ and $H'_1$ have changed and we denoted them here by $Q'_2$ and $H'_2$.

For arbitrary $T>0$, one can cover the interval $[0,T]$ by intervals of the form $[kT_0,(k+1)T_0],\;k\in\N_0,$ in finitely many steps (say in $l$ steps). So we have
$$\sum_{\alpha\in \mathcal{I}}\sup_{t\in[0,T]}\|u_\alpha(t)\|_X^2(2\N)^{-p\alpha}\leq \frac{Q+G}{\frac{1}{3}-C(T_0)K^2},$$
where $Q=\max_{1\leq k\leq l}\{Q_k\}$. Thus,
$$U(t,\omega)=\sum_{\alpha\in \mathcal{I}}u_\alpha(t)H_\alpha(\omega)\in C([0,T],X)\otimes (S)_{-1,-p}.$$
Also,
$$\sum_{\alpha\in \mathcal{I}}\sup_{t\in[0,T]}\|\frac{d}{dt}u_\alpha(t)\|_X^2(2\N)^{-p\alpha} \leq \frac{Q'+G'+H'+N'}{\frac{1}{5}-C(T_0)K^2},$$
where $Q'=\max_{1\leq k\leq l}\{Q'_k\}$, $H'=\max_{1\leq k\leq l}\{H'_k\}$. Since $\frac{d}{dt}u_\alpha(t)\in C([0,T],X)$, $\alpha\in\mathcal I$, we have
$$\frac{d}{dt}U(t,\omega)=\sum_{\alpha\in \mathcal{I}}\frac{d}{dt}u_\alpha(t)H_\alpha(\omega)\in C([0,T],X)\otimes (S)_{-1,-p}.$$
Therefore, $U(t,\omega)\in C^1([0,T],X)\otimes (S)_{-1,-p}$ and
thus, $U$ is a solution of (\ref{nasaJNA}) in the sense of Definition \ref{U resaca CP}.

The solution $U$ is unique due to the uniqueness of the coordinatewise (classical) solutions $u_\alpha$ in (\ref{S1}) and due to uniqueness in the Wiener-It\^o chaos expansion.

\end{proof}

Note that according to the previous theorem the solution $U$ remains in the same stochastic order space $(S)_{-1,-p}$ where the input data $U^0,\; \mathbf{A}U^0$ and $F$ belong to.

\begin{example} We provide three examples  of equation \eqref{nasaJNA} where $\mathbf A$ is a uniformly bounded (not a simple) coordinatewise operator.
Consider the Banach space $X=L^p(\R),\;1\leq p<\infty,$  and the stochastic Cauchy problem
\begin{equation}\label{ex1}\begin{split}\frac{d}{dt}U(t,x,\omega) &= \mathbf{A}U(t,x,\omega) + W\lozenge U(t,x,\omega) + F(t,x,\omega),\\
U(0,x,\omega)&=U^0(x,\omega),\end{split}\end{equation}
where the operator $\mathbf{A}:Dom(\mathbf{A})\to X\otimes(S)_{-1}$ is a coordinatewise operator composed out of a family of closed operators $\{A_\alpha\}_{\alpha\in \mathcal{I}}$ of the form $A_\alpha=a_\alpha D,\;\alpha\in \mathcal{I}$, where the functions $a_\alpha\in L^\infty(\R),\;\alpha\in \mathcal{I},$ are uniformly bounded, i.e. $\sup_{x\in \R}|a_\alpha(x)|\leq M,\;\alpha\in \mathcal{I},$ for some $M>0,$ and  $D$ is one of the following differential operators: $\frac{\partial}{\partial x},\;\frac{\partial^2}{\partial x^2}$ or $\frac{\partial^2}{\partial x^2}+\frac{\partial}{\partial x}$,
and $W=\sum_{k\in \N}\xi_k H_{\varepsilon_k}$ represents spatial white noise.
Then, (\ref{ex1}) is equivalent to the infinite system
\begin{equation*}\begin{split}\frac{d}{dt} u_\alpha(t,x) &= A_\alpha u_\alpha(t,x) +\sum_{k\in \N}\xi_k(x)u_{\alpha-\varepsilon_k}(t,x)  + f_\alpha(t,x)\\
u_\alpha(0,x)&=u_\alpha^0(x),\quad\alpha\in \mathcal{I}.\end{split}\end{equation*}

The $C_0-$semigroup that corresponds to the closed operator $D$, denoted by $T_t,\;t\geq 0,$ is, respectively,
\begin{align*}T_tg(x)&=g(t+x),\;\;g\in L^p(\R), \qquad \mbox{for }\; D=\frac{\partial}{\partial x},\\
T_tg(x)&=\frac{1}{\sqrt{4\pi t}}\int_{\R}g(x-y)e^{-\frac{y^2}{4t}}dy,\;\;g\in L^p(\R),\qquad \mbox{for }\; D=\frac{\partial^2}{\partial x^2},\\
T_tg(x)&=\frac{1}{\sqrt{4\pi t}}\int_{\R}g(x-y)e^{-\frac{(y+t)^2}{4t}}dy,\;\;g\in L^p(\R),\qquad \mbox{for }\; D=\frac{\partial^2}{\partial x^2}+\frac{\partial}{\partial x}.
\end{align*}
In all cases, we have, using Young's inequality, that $\|T_t\|\leq 1,\;t\geq 0.$ The $C_0-$semigroups corresponding to the operators $A_\alpha,\;\alpha\in \mathcal{I},$ are of the form $(S_t)_\alpha=a_\alpha T_t$. Thus, $\|(S_t)_\alpha\|\leq M,\;\alpha\in \mathcal{I}$. The operators $B_\alpha,\;\alpha\in \mathcal{I},$ are given by $B_{\varepsilon_k}=\xi_k,\;k\in\N$ and $B_\alpha=0,\;\alpha\neq\varepsilon_k.$ Thus, $\|B_\alpha\|\leq\sup_{k\in\mathbb N}\|\xi_k\|_{L^\infty(\mathbb R)}\leq 1,\;\alpha\in \mathcal{I}$. Now, according to Theorem \ref{th polugr}, equation \eqref{ex1} has a
unique solution $U(t,x,\omega)=\sum_{\alpha\in \mathcal{I}}u_\alpha(t,x)H_\alpha(\omega),$ where
$$u_\alpha(t,x)=(S_t)_\alpha u_\alpha^0(x)+\int_0^t (S_{t-s})_\alpha(\sum_{k}\xi_k(x)u_{\alpha-\varepsilon_k}(s,x)+f_\alpha(s,x))ds,\;\alpha\in \mathcal{I}.$$
\end{example}

\begin{example} Consider the Cauchy problem
\begin{equation*}\begin{split}\frac{d}{dt}U(t,\omega) &= \mathbf{A}U(t,\omega) + \mathbf{B}\lozenge U(t,\omega) + F(t,\omega)\\
U(0,\omega)&=U^0(\omega),\end{split}\end{equation*}
where $\mathbf  A$ is a simple coordinatewise operator $A_\alpha=A$, $\alpha\in\mathcal I$, generating a $C_0-$semigroup, $B_\alpha\neq 0$ only for $\alpha=\varepsilon_k,\;k\in \N,$ are such that $\sum_{k\in\N}\|B_{\varepsilon_k}\|(2k)^{-\frac{p}{2}}<\infty$, and $U^0$ and $F$ are deterministic functions, i.e. $u^0_\alpha= 0$ and $f_\alpha= 0$ for all $\alpha\in\mathcal I\setminus\{\mathbf 0\}$.

The solution of this system, according to Theorem \ref{th polugr}, is
\begin{equation*}\begin{split}u_\mathbf{0}(t)&=T_tu^0_\mathbf{0}+\int_0^t T_{t-s}f_{\mathbf{0}}(s)ds,\\
u_\alpha(t)&=\int_0^tT_{t-s}\Big(\sum_{k\in \N}B_{\varepsilon_k}u_{\alpha-\varepsilon_k}(s)\Big)ds,\quad \alpha\in \mathcal{I}\setminus{\mathbf 0},\end{split}\end{equation*}
the same form as it was obtained in \cite{Boris}.
\end{example}

We provide two generalisations of Theorem \ref{th polugr}: one possibility is to allow the operators $B_\alpha$ to depend on the time variable $t$ (except for $B_{\mathbf 0}$ which must be free of $t$). This embraces for example SPDEs driven by space-time noises which have zero expectation (and are thus free of $t$). The other possibility is to allow $B_{\mathbf 0}$ to be unbounded but satisfying certain properties so that $A_\alpha+B_{\mathbf 0}$ are infinitesimal generators of $C_0-$semigroups. For example, if $A_\alpha=\frac{\partial^2}{\partial x^2}$ and $B_{\mathbf 0}=\frac{\partial}{\partial x}$, then although $B_0$ is unbounded, $A_\alpha+B_{\mathbf 0}$ is the generator of a contraction semigroup. Following \cite{Nagel} we will enlist some sufficient conditions which ensure that $A_\alpha+B_{\mathbf 0}$ is the  generators of a $C_0-$semigroup.

\begin{remark} In Theorem \ref{th polugr} one can consider operators $B_\alpha(t),\;\alpha\in \mathcal{I}\setminus\{\mathbf{0}\}$, depending on $t$, so that $B_\alpha\in C^1([0,T],L(X)),\;\alpha\in \mathcal{I}\setminus \{\mathbf{0}\}$, $B_\mathbf{0}(t)=B_\mathbf{0}\in L(X),$ for all $t\in [0,T],$ and
\begin{align}\label{ogr B(t)}
K:&=\sum_{\alpha\in \mathcal{I},\atop \alpha>\mathbf 0}\|B_\alpha\|_{C^1([0,T],L(X))}(2\N)^{-p\frac{\alpha}{2}}\nonumber\\
&=\sum_{\alpha\in \mathcal{I},\atop \alpha>\mathbf 0}\Big(\sup_{t\in[0,T]}\|B_\alpha(t)\|_{L(X)}+\sup_{t\in[0,T]}\|\frac{d}{dt}B_\alpha(t)\|_{L(X)}\Big)(2\N)^{-p\frac{\alpha}{2}}<\infty.
\end{align}
Replacing (\ref{ogr B}) by (\ref{ogr B(t)}) and retaining all other assumptions of Theorem \ref{th polugr}, one can again obtain a unique solution $U$ in $C^1([0,T],X)\otimes (S)_{-1,-p}$ of the corresponding Cauchy problem (\ref{nasaJNA}).

The solution is $U(t,\omega)=\sum_{\alpha\in \mathcal{I}}u_\alpha(t)H_\alpha(\omega),$ $u_\alpha(t) \in C^1([0,T],X)$, $ \alpha\in \mathcal{I}$,
where (see (\ref{S1}))
\begin{equation}\label{S1(t)}\begin{split} u_\mathbf{0}(t)&= (S_t)_\mathbf{0} u_\mathbf{0}^0+ \int_0^t (S_{t-s})_\mathbf{0} f_\mathbf{0}(s)ds,\quad t\in[0,T] ,\\
u_\alpha(t)&= (S_t)_\alpha u_\alpha^0+  \int_0^t (S_{t-s})_\alpha \Big(\sum_{\mathbf{0}<\beta\leq \alpha}B_\beta(s) u_{\alpha-\beta}(s)+ f_\alpha(s)\Big)ds,\quad t\in[0,T].\\
\end{split}\end{equation}

Its derivative is $\frac{d}{dt}U(t,\omega)=\sum_{\alpha\in \mathcal{I}}\frac{d}{dt}u_\alpha(t)H_\alpha(\omega),$ where (see (\ref{DS1}))
\begin{equation}\label{DS1{t}}\begin{split} \frac{d}{dt}u_\mathbf{0}(t)&= (S_t)_\mathbf{0}(A_\mathbf{0}+B_\mathbf{0})u_\mathbf{0}^0+ \int_0^t (S_{t-s})_\mathbf{0}\frac{d}{ds} f_\mathbf{0}(s)ds+(S_t)_\mathbf{0}f(0),\quad t\in[0,T] ,\\
\frac{d}{dt}u_\alpha(t)&= (S_t)_\alpha (A_\alpha+B_\mathbf{0})u_\alpha^0\\
&+  \int_0^t (S_{t-s})_\alpha \Big(\sum_{\mathbf{0}<\beta\leq \alpha}\Big(B_\beta(s) \frac{d}{ds} u_{\alpha-\beta}(s)+\frac{d}{ds}B_\beta (s) u_{\alpha-\beta}(s)\Big)+ \frac{d}{ds}f_\alpha(s)\Big)ds \\
& +(S_t)_\alpha\Big(\sum_{\mathbf{0}<\beta\leq \alpha}B_\beta(0) u_{\alpha-\beta}(0)+f_\alpha(0)\Big) ,\quad t\in[0,T],\quad \alpha\in \mathcal{I}.\end{split}\end{equation}
The proof can be performed in the same manner as in Theorem \ref{th polugr}, now taking $T_0\in (0,T]$ to be small enough so that $C(T_0)<\frac{1}{6K^2},$ since now we have six summands in (\ref{DS1{t}}) instead of the previous five in (\ref{DS1}).
\end{remark}

\begin{remark}
In Theorem \ref{th polugr} one can consider the operator $B_{\mathbf 0}$ to be unbounded, densely defined on $D$ (the same domain which is common for all $A_\alpha$) so that either of the following holds:
\begin{itemize}

\item[(i)] $A_\alpha$, $\alpha\in\mathcal I$, are generating contraction semigroups (i.e. $M=1$, $w=0$), and $B_{\mathbf 0}$ is dissipative, $A_\alpha-$bounded with $a_\alpha^0<1$ (i.e. there exist $a_\alpha,b_\alpha>0$ such that $\|B_{\mathbf 0}x\|\leq a_\alpha\| A_\alpha x\|+b_\alpha \|x\|$, $x\in D$, and $a_\alpha^0=\inf\{a_\alpha>0: \exists b_\alpha>0, \forall x\in D, \|B_{\mathbf 0}x\|\leq a_\alpha\| A_\alpha x\|+b_\alpha \|x\|\}$), for all $\alpha\in\mathcal I$,

\item[(ii)] $B_{\mathbf 0}$ is closable, dissipative and $A_\alpha-$compact (i.e. $B:(D,\|\cdot\|_{A_\alpha})\rightarrow X$ is compact where $\|\cdot\|_{A_\alpha}$ denotes the graph norm), for all $\alpha\in\mathcal I$,

\item[(iii)] $A_\alpha$ are generating analytic semigroups (i.e. $w<0$), $\alpha\in\mathcal I$, and $B_{\mathbf 0}$ is closable and $A_\alpha-$compact .
\end{itemize}
Then, $A_\alpha+B_{\mathbf 0}$ is the infinitesimal generator of a $C_0-$semigroup  (denote it $(S_t)_\alpha$) for all $\alpha\in\mathcal I$. If the semigroups  $(T_t)_\alpha$ corresponding to $A_\alpha$ are uniformly bounded in $\alpha$, then so will be $(S_t)_\alpha$. Retaining all other assumptions of Theorem \ref{th polugr}, now we follow the same proof pattern with the semigroup $(S_t)_\alpha$, $\|(S_t)_\alpha\|\leq \tilde M e^{\tilde w t}$, for some $\tilde M\geq1$, $\tilde w\in\mathbb R$, independent of $\alpha$.

Finally we note that in case (i) and (ii) $A_\alpha+B_{\mathbf 0}$ will be generating contraction semigroups, while in case (iii) they will be generating analytic semigroups.

\end{remark}

\section{Stationary equations}

In this section we consider stationary equations of
 the form
\begin{equation}
\label{elipt jedn} \mathbf{A} U + \mathbf{B} \lozenge U + F = 0,
\end{equation}
 where $\mathbf A:\,  X\otimes (S)_{-1} \rightarrow X\otimes (S)_{-1}$ and $\mathbf B \lozenge:\, X\otimes (S)_{-1} \rightarrow X\otimes (S)_{-1}$ are  coordinatewise operators as in (\ref{delovanje A}) and (\ref{delovanje B}). We assume that $\{A_\alpha\}_{\alpha\in \mathcal{I}} $ and $\{B_\alpha\}_{\alpha\in \mathcal{I}} $ are bounded operators and that $A_\alpha$  are of the form
\begin{equation}A_\alpha = \widetilde{A}_\alpha +  C_\alpha,\;\;\;\alpha \in\mathcal I,\nonumber\end{equation} where  $B_\mathbf{0}$ and $\widetilde{A}_\alpha$, $\alpha\in \mathcal{I}$
are compact  operators and $C_\alpha$ are self adjoint operators for all $\alpha\in \mathcal{I}$. Denote by $r_\alpha$ the eigenvalue corresponding to the orthogonal family of eigenvectors $H_\alpha$, i.e.  $C_\alpha (H_\alpha) = r_\alpha H_\alpha$, $\alpha\in \mathcal{I}$. Using classical results of elliptic PDEs and the Fredholm alternative (see \cite{Gilbarg})  we prove existence and uniqueness of the solution to \eqref{elipt jedn}.

\begin{theorem}
\label{Fredholm appl} Let $X$ be a Banach space. Let
$\mathbf A:\,  X\otimes (S)_{-1} \rightarrow X\otimes (S)_{-1}$ and $\mathbf B \lozenge:\, X\otimes (S)_{-1} \rightarrow X\otimes (S)_{-1}$ be coordinatewise operators, for which the following assumptions hold:
\begin{enumerate}

\item $\mathbf{A}$ is of the form $\mathbf{A}=\mathbf{\widetilde{A}}+ \mathbf{C}$, where  $\mathbf{\widetilde{A}}(U) = \sum\limits_{\alpha\in \mathcal{I}} \widetilde{A}_\alpha (u_\alpha) H_\alpha$ and $\widetilde{A}_\alpha : X \rightarrow X$ are compact operators for all $\alpha\in\mathcal I$, $\mathbf{C} (U)= \sum\limits_{\alpha\in \mathcal{I}}\,  r_\alpha u_\alpha H_\alpha$,  $r_\alpha\in\mathbb R$, $\alpha\in\mathcal I$, and $\mathbf B$ is of the form (\ref{delovanje B}), where $B_{\mathbf 0}:X\rightarrow X$ is a compact operator. Assume there exists $K>0$ such that:
\begin{equation}\label{uslov za ralpha}
 - \|\widetilde{A}_\alpha\| - \|B_\mathbf{0}\|-r_\alpha\geq 0, \quad\mbox{ for all }\;\;\alpha\in\mathcal I,
\end{equation}
and
\begin{equation}
\label{uslov za operator a} \sup\limits_{\alpha\in
\mathcal{I}}\left(\frac{1}{- r_\alpha - \|\widetilde{A}_\alpha\| - \|B_\mathbf{0}\|}\right) <  K.
\end{equation}

\item $\mathbf{B}$ is of the form (\ref{delovanje B}), where $B_\beta:X\rightarrow X$, $\beta\in\mathcal I\setminus\{\mathbf 0\}$, are bounded operators and
there
exists $p>0$ such that
\begin{equation}
\label{uslov koren iz dva} K\sum\limits_{\beta\in\mathcal I \atop\beta>\mathbf 0}\|B_\beta\| \,
(2\mathbb{N})^{\frac{-p\beta}{2}} < \frac{1}{\sqrt{2}}.
\end{equation}

\item For every $\alpha\in \mathcal{I}$
\begin{equation}\label{jezgro} {\rm Ker} \left(\widetilde{A}_\alpha + (1+ r_\alpha) {\rm Id} + B_\mathbf{0}\right) = \{0\}.
\end{equation}

\end{enumerate}

Then, for every $F\in X\otimes (S)_{-1,-p}$ there
exists a unique solution $U\in X\otimes (S)_{-1,-p}$ to equation
\eqref{elipt jedn}.
\end{theorem}

\begin{proof}
Equation \eqref{elipt jedn} is equivalent to $ U - (\mathbf{\widetilde{A}}(U)+ \mathbf{C} U + U
+\mathbf{B}\lozenge U)= F$ and
$$\sum\limits_{\gamma\in \mathcal{I}}\left( u_\gamma - \widetilde{A}_\gamma u_\gamma- (1+ r_\gamma)\,  u_\gamma-
\sum\limits_{\alpha+\beta =\gamma} B_\alpha (u_{\beta})
\right) H_\gamma= \sum\limits_{\gamma\in \mathcal{I}} f_\gamma H_\gamma.$$ Due to uniqueness of the Wiener-It\^ o
chaos expansion this is equivalent to
\begin{equation}\label{ugamma}
 u_\gamma - \left( \widetilde{A}_\gamma + (1+ r_\gamma)Id + B_\mathbf{0} \right) u_\gamma =  f_\gamma + \sum\limits_{\mathbf 0<\beta\leq\gamma} B_\beta (u_{\gamma-\beta}),
\quad \gamma\in \mathcal{I}.
\end{equation}
By  \eqref{jezgro} it follows that for each $\gamma\in \mathcal{I}$ the homogeneous equation  $$  u_\gamma - \left( \widetilde{A}_\gamma + (1+ r_\gamma) Id + B_\mathbf{0} \right) u_\gamma =  0$$ has only trivial solution $u_\gamma=0$. Since the
operator $\widetilde{A}_\gamma + (1+ r_\gamma) Id + B_\mathbf{0} $ is compact, the classical Fredholm alternative implies
that for each $\gamma\in \mathcal{I}$ there exists a unique $u_\gamma$ that solves \eqref{ugamma}
and it is of the form
\begin{equation*}u_\gamma=  (Id - (( r_\gamma +1)\, Id + \widetilde{A}_\gamma + B_\mathbf{0}) )^{-1}
\left(f_\gamma+ \sum\limits_{\beta>\mathbf0} B_\beta
(u_{\gamma-\beta})\right),\, \, \,\quad \gamma\in \mathcal{I},\end{equation*}
so that
$$\|u_\gamma\|_X \leq \frac{1}{- r_\gamma - \|\widetilde{A}_\gamma\| - \|B_\mathbf{0}\|} \cdot \left( \|f_\gamma \|_X +
\sum\limits_{\beta>\mathbf0} \|B_\beta\|
\|u_{\gamma-\beta}\|_X \right), \quad \gamma\in \mathcal{I}.$$ We
will prove that $\sum\limits_{\gamma\in\mathcal I}u_\gamma\otimes
H_\gamma$ converges in $X\otimes(S)_{-1}$. Indeed,

\begin{eqnarray*}
\sum\limits_{\gamma\in\mathcal I}\|u_\gamma\|_X^2(2\mathbb N)^{-p\gamma}&\leq&
K^2\sum\limits_{\gamma\in\mathcal I}\left(\|f_\gamma\|_X +
\sum_{\gamma=\alpha+ \beta, \alpha>\mathbf{0}}\,
\|B_\alpha\| \|u_{\beta}\|_X \right)^2(2\mathbb
N)^{-p\gamma}\\
&\leq& 2 K^2\left(\sum_{\gamma\in\mathcal I}\|f_\gamma\|_X^2 (2\mathbb N)^{-p\gamma} +\sum_{\gamma\in\mathcal I}(\sum_{\gamma=\alpha+\beta, \alpha>\mathbf{0}}
\|B_\alpha\|\|u_\beta\|_X)^2  (2\mathbb N)^{-p\gamma}\right)\\
&\leq& 2 K^2\left(\sum_{\gamma\in\mathcal I}\|f_\gamma\|_X^2 (2\mathbb N)^{-p\gamma}+(\sum_{\alpha>\mathbf{0}}
\|B_\alpha\|(2\mathbb
N)^{-\frac{p\alpha}2})^2\sum\limits_{\beta\in\mathcal
I}\|u_\beta\|_X^2(2\mathbb N)^{-p\beta}\right).
\end{eqnarray*}

Therefore,
$$(1-2K^2(\sum\limits_{\alpha>\mathbf{0}}\|B_\alpha\|(2\mathbb N)^{-\frac{p\alpha}2})^2)\cdot \sum\limits_{\gamma\in\mathcal I}
\|u_\gamma\|_X^2(2\mathbb N)^{-p\gamma}\leq
2K^2\sum_{\gamma\in\mathcal I}\|f_\gamma\|_X^2(2\mathbb
N)^{-p\gamma}.$$ By assumption \eqref{uslov koren iz dva} we have that $M=1- 2K^2
(\sum\limits_{\alpha>\mathbf{0}}\|B_\alpha\|(2\mathbb
N)^{-\frac{p\alpha}2})^2 >0$. This implies
$$\sum_{\gamma\in\mathcal I}\|u_\gamma\|_X^2(2\mathbb N)^{-p\gamma}\leq \frac{2K^2}{M} \sum_{\gamma\in\mathcal I}\|f_\gamma\|_X^2(2\mathbb
N)^{-p\gamma}<\infty.$$
\end{proof}

\begin{example}\rm
We provide some special cases of equation \eqref{elipt jedn}.
\begin{itemize}
\item[1.] If $A_\alpha=0$ for all $\alpha\in \mathcal{I}$ and $B_\alpha$, $\alpha\in \mathcal{I}$ are second order strictly elliptic partial differential operators in divergent form
    \begin{equation}\label{divergent form}
    B_\alpha=\sum_{i=1}^n D_i(\sum_{j=1}^n
a^{ij}_\alpha(x)D_j+b^i_\alpha(x))+\sum_{i=1}^n
c^i_\alpha(x)D_i+d_\alpha(x)
    \end{equation}
     with essentially bounded coefficients, then  equation \eqref{elipt jedn} reduces to the elliptic equation
    \begin{equation} \mathbf{B} \lozenge U = F, \nonumber
    \end{equation}which was solved in \cite{ps} and \cite{DP2}.

\item[2.]  Let $ \widetilde{A}_\alpha = 0$ for all $\alpha\in \mathcal{I}$
and let  $B_\alpha$, $\alpha\in \mathcal{I}$, be second order strictly elliptic partial differential operators in divergent form \eqref{divergent form}. Let $\mathbf{C}=c \, P(\mathcal{R})$, for some $c\in \mathbb{R}$, where  $\mathcal{R}$ is the Ornstein-Uhlenbeck operator, $P$ a polynomial of degree $m$ with real coefficients and $P(\mathcal{R})$ the differential operator $P(\mathcal{R})=p_m
\mathcal{R}^m+p_{m-1}\mathcal{R}^{m-1}+...+p_1\mathcal{R}+p_0 Id$. Then, the corresponding eigenvalues are  $r_\alpha = c P (|\alpha|)$, $\alpha\in \mathcal{I}$.
Hence, equation \eqref{elipt jedn} transforms to the elliptic  equation with a perturbation term driven by the polynomial of the Ornstein-Uhlenbeck operator
\begin{equation*}
\mathbf{B}\lozenge U + c P(\mathcal{R})U  = F,
\end{equation*}
that was solved in \cite{JSAA}.
\end{itemize}
\end{example}


\section*{Acknowledgement}

The paper was supported by the projects \textit{Modeling and harmonic
analysis methods and PDEs with singularities}, No. 174024, and
\textit{Modeling and research methods of operational control of traffic
based on electric traction vehicles optimized by power consumption
criterion}, No. TR36047, both financed by the Ministry of Science,
Republic of Serbia and project No. 114-451-3605/2013 financed by
the Provincial Secretariat for Science of Vojvodina.


\begin{thebibliography}{99}

\bibitem{Dejv}
Applebaum, D.: On the infinitesimal generators of Ornstein-Uhlenbeck
processes with jumps in Hilbert spaces. \emph{Potential Anal.} \textbf{26}, (2007),
79--100.

\bibitem{Bog}
Bogachev, V. I.: Differentiable Measures and the Malliavin Calculus.
\emph{American Mathematical Society}, Providence, 2010.


\bibitem{Cat2}
Catuogno, P., Olivera, C.: On stochastic generalized functions. \emph{Infin.
Dimens. Anal. Quantum Probab. Relat. Top.} \textbf{14(2)}, (2011), 237--260.

\bibitem{Nagel}
Engel, K.J., Nagel, R.: One-Parameter Semigroups for Linear Evolution Equations. \emph{Springer-Verlag}, New York, 2000.

\bibitem{Gilbarg}
Gilbarg, D., Trudinger, N.S.: Elliptic Partial Differential Equations
of Second Order. \emph{Springer-Verlag}, Berlin, 1998.

\bibitem{Hida}
Hida, T., Kuo, H.-H., Pothoff, J., Streit, L.: White
Noise. An Infinite-dimensional Calculus. \emph{Kluwer Academic
Publishers Group}, Dordrecht, 1993.

\bibitem{HOUZ}
Holden, H., \O ksendal, B., Ub\o e, J., Zhang, T.: Stochastic Partial Differential Equations. A
Modeling, White Noise Functional Approach. Second Edition. \emph{Springer}, New York, 2010.

\bibitem{yuhu}
Hu, Y.: Chaos expansion of heat equations with
white noise potentials. \emph{Potential Anal.} \textbf{16}, (2002), 45--66.

\bibitem{grci}
Kalpinelli, E., Frangos, N., Yannacopoulos, A.: A Wiener chaos approach to hyperbolic SPDEs.
\emph{Stoch. Anal. Appl.} \textbf{29}, (2011), 237--258.

\bibitem{AADM}
Levajkovi\' c, T., Pilipovi\'c, S., Sele\v si, D.: Fundamental
equations with higher order Malliavin operators. \emph{Stochastics: An International Journal of Probability and Stochastic Processes}, accepted for publication.

\bibitem{JSAA}
Levajkovi\' c, T., Pilipovi\'c, S., Sele\v si, D.: The stochastic
Dirichlet problem driven by the Ornstein-Uhlenbeck operator:
Approach by the Fredholm alternative for chaos expansions.
\emph{Stoch. Anal. Appl.} \textbf{29}, (2011), 317-331.

\bibitem{Boris}
Lototsky, S., Rozovskii, B.: Stochastic partial
differential equations driven by purely spatial noise. \emph{SIAM J.
Math. Anal.} \textbf{41(4)}, (2009), 1295--1322.

\bibitem{Boris2}
Lototsky, S., Rozovskii, B.: Bilinear stochastic
elliptic equations. \emph{Quad. Mat.} \textbf{25}, (2010), 207--221.

\bibitem{Irina}
Melnikova, I.V., Alshanskiy, M.A.: Regularized
and generalized solutions of infinite-dimensional stochastic
problems. \emph{Sb. Math.} \textbf{202(11)}, (2011), 1565--1592.

\bibitem{Irina2}
Melnikova, I.V., Alshanskiy, M.A.: Generalized
solutions of abstract stochastic problems. \emph{Oper. Theory
Adv. Appl.} \textbf{231}, (2013), 341--352.

\bibitem{paz}
Pazy, A.: Semigroups of Linear Operators and Applications to Partial
Differential Equations. \emph{Springer-Verlag}, New York, 1983.


\bibitem{GRPW}
Pilipovi\'c, S., Sele\v si, D.: Expansion theorems for generalized random processes,
Wick products and applications to stochastic differential
equations.  \emph{Infin. Dimens. Anal. Quantum Probab. Relat. Top.}
\textbf{10(1)}, (2007), 79--110.

\bibitem{ps}
Pilipovi\'c, S., Sele\v si, D.:  On the generalized stochastic
Dirichlet problem - Part I: The stochastic weak maximum principle.
\emph{Potential Anal.} \textbf{32}, (2010), 363-387.

\bibitem{DP2}
Pilipovi\' c, S., Sele\v si, D.: On the generalized stochastic Dirichlet problem
-- Part II: Solvability, stability and the Colombeau case. \emph{Potential
Anal.} \textbf{33}, (2010), 263-289.

\bibitem{Proske}
Proske, F.: The stochastic transport equation driven by L\'evy white
noise. \emph{Commun. Math. Sci.} \textbf{2(4)}, (2004), 627--641.

\bibitem{FundSol}
Sele\v si, D.: Fundamental solutions of singular SPDEs. \emph{Chaos
Solitons Fractals} \textbf{44}, (2011), 526--537.


\end{thebibliography}
\end{document}